\pdfoutput=1
\documentclass[11pt]{article}

\usepackage[a4paper,%
  outer=25mm,%
  inner=25mm,%
  top=20mm,%
  bottom=20mm%
]{geometry}

\usepackage[T1]{fontenc}
\usepackage[utf8]{inputenc}
\usepackage{lmodern}
\usepackage[%
  backend=biber,%
  language=english,%
  autolang=other,%
  abbreviate=true,%
	date=year,%
  isbn=false,%
  doi=false,%
	url=false,%
  giveninits=true%
]{biblatex}
\renewbibmacro{in:}{}

\usepackage{amssymb}
\usepackage{amsfonts}
\usepackage{amsthm}
\usepackage{stmaryrd}
\usepackage{mathtools}
\usepackage{calc}
\usepackage{cancel}
\usepackage{tikz-cd}
\usepackage{authblk}

\usepackage{enumitem}

\usepackage[bookmarks]{hyperref}

\title{Bi-Lagrangian structures and the space of rays}
\author[$\dagger$]{Wojciech Domitrz}
\author[$\star$]{Marcin Zubilewicz}
\affil[$\dagger$]{\footnotesize
Faculty of Mathematics and Information Science,
Warsaw University of Technology,

\textit{e-mail:} wojciech.domitrz@pw.edu.pl
}
\affil[$\star$]{\footnotesize
Faculty of Mathematics and Information Science,
Warsaw University of Technology,

\textit{e-mail:} marcin.zubilewicz.dokt@pw.edu.pl
}
\date{}

\theoremstyle{plain}
\newtheorem{thm}{Theorem}%[section]
\newtheorem{lem}[thm]{Lemma}
\newtheorem{prop}[thm]{Proposition}

\theoremstyle{definition}
\newtheorem{defn}[thm]{Definition}
\newtheorem{exmp}{Example}%[section]
\theoremstyle{remark}

\numberwithin{equation}{section}

\let\oldsection\section%
\renewcommand{\section}{%
  \renewcommand{\theequation}{\thesection.\arabic{equation}}%
  \oldsection}%

\graphicspath{{.}}

\makeatletter
\newcommand{\mask}[2]{{\mathpalette\mask@{{#1}{#2}}}}
\newcommand{\mask@}[2]{\mask@@{#1}#2}
\newcommand{\mask@@}[3]{%
  \settowidth{\dimen@}{$\m@th#1#2$}%
  \makebox[\dimen@]{$\m@th#1#3$}%
}
\makeatother % snippet by egreg (https://tex.stackexchange.com/questions/548426/how-to-make-text-appear-in-the-middle-of-a-phantom-in-math-mode) (requires 'calc')

\newcommand{\ifstr}[2]{\def\cmpstrone{#1}\def\cmpstrtwo{#2}%
\ifx\cmpstrone\cmpstrtwo}

{
  \par\medskip
}

\newcommand{\Rb}[1]{\ifstr{#1}{1}\mathbb{R}\else\mathbb{R}^{#1}\fi}

\newcommand{\maps}[3]{#1:#2\rightarrow\,#3}
\newcommand{\set}[1]{\left\lbrace#1\right\rbrace}
\newcommand{\smset}[1]{\textstyle\lbrace#1\rbrace}

\newcommand{\smfrac}[2]{{\textstyle\frac{#1}{#2}}}
\newcommand{\basis}[1]{{\textstyle{\smash{\frac{\partial}{\partial#1}}}}}
\newcommand{\smsum}[1]{\smash{\textstyle\sum_{#1}}\,}
\newcommand{\pd}[2]{\frac{\partial #1}{\partial #2}}
\newcommand{\smpd}[2]{\textstyle\frac{\partial #1}{\partial #2}}
\newcommand{\pdd}[3]{\frac{\partial^2 #1}{\partial #2\,\partial #3}}
\newcommand{\smpdd}[3]{{\textstyle\frac{\partial^2 #1}{\partial #2\,\partial #3}}}

\newcommand{\bm}[1]{\begin{bmatrix} #1 \end{bmatrix}}

\DeclarePairedDelimiter\abs{\lvert}{\rvert}
\DeclarePairedDelimiter\norm{\lVert}{\rVert}

\newcommand{\fol}{\mathcal{F}}
\newcommand{\gol}{\mathcal{G}}
\newcommand{\hol}{\mathcal{H}}
\newcommand{\pol}{\mathcal{P}}
\newcommand{\qol}{\mathcal{Q}}
\newcommand{\web}{\mathcal{W}}
\newcommand{\II}{I\mkern-4mu I}

\newcommand*{\Rc}{\mathrm{Rc}}

\newcommand{\Fols}{\operatorname{Fol}}
\newcommand{\codim}{\operatorname{codim}}

\begin{document}

  \maketitle
  \begin{abstract}
    This paper focuses on local curvature invariants associated with
    bi-Lagrangian structures. We establish several geometric conditions that
    determine when the canonical connection is flat, building on our previous
    findings regarding divergence-free webs \cite{dfw}. Addressing questions
    raised by Tabachnikov \cite{2-webs}, we provide complete solutions to two
    problems: the existence of flat bi-Lagrangian structures within the space
    of rays induced by a pair of hypersurfaces, and the existence of flat
    bi-Lagrangian structures induced by tangents to Lagrangian curves in the
    symplectic plane.
  \end{abstract}

  \tableofcontents

\section{Introduction}

\emph{Bi-Lagrangian structures} \cite{etayosantamaria, bilagrangian, vaisman},
(also: \emph{bi-Lagrangian manifolds} or \emph{Lagrangian $2$-webs}
\cite{2-webs}) are quadruples $(M,\omega,\fol,\gol)$, where $(M,\omega)$ is a
symplectic manifold and $\fol,\gol$ are foliations of $M$ into leaves that are
Lagrangian with respect to the ambient symplectic form $\omega$ and intersect
transversely at each point $p\in M$.

A wealth of examples of such structures comes from mathematical physics. For
instance, in geometric optics one considers the space of all oriented rays
$\ell$ inside a uniform medium $\Rb{n}$, which are represented by pairs
$(q,p)\in T^*S^{n-1}$ with point $q=v/\norm{v}\in S^{n-1}$ and the cotangent
vector $p$ given by $p(w) = \langle w,x\rangle - \langle x,v\rangle\langle
w,v\rangle/\norm{v}^2$, where $v\in\Rb{n}$ is the direction vector of $\ell$,
the point $x\in\Rb{n}$ is any point on $\ell$, and $\langle u,w\rangle =
\smsum{i}u_iw_i$ is the standard inner product. Each point $x\in\Rb{n}$
determines a submanifold of all rays crossing that point. A classical result
states that this submanifold is Lagrangian with respect to the canonical
symplectic form $\sigma$ on $T^*S^{n-1}$ \cite[49]{arnoldsympl}. This leads
us to consider a pair of hypersurfaces $H,K\subseteq\Rb{n}$ and a ray $(q,p)\in
T^*S^{n-1}$ crossing $H$ and $K$ transversely. Since each ray close to $(q,p)$
also has a transverse intersection with both hypersurfaces, one obtains a pair
of Lagrangian foliations $\fol,\gol$ of a neighbourhood $M$ of the ray $(q,p)$
inside $T^*S^{n-1}$. It turns out that for generic data these two foliations
are complementary and $(M,\sigma,\fol,\gol)$ is a bi-Lagrangian structure on
$M\subseteq T^*S^{n-1}$ \cite{2-webs}.

Due to the fact that each symplectic form $\omega$ on a $2n$-dimensional
manifold provides us with canonical volume form $\omega^n$, it is possible to
view bi-Lagrangian manifolds through the lens of unimodular geometry, or the
geometry of volume-preserving transformations. This reduction of the ambient
geometry leads naturally to structures of the kind $(M,\Omega,\fol,\gol)$,
where $\Omega$, instead of being a symplectic form, is a volume form. The study
of these objects, called \emph{divergence-free $2$-webs} in \cite{2-webs}, and
their generalization to higher number of foliations of arbitrary codimension
\cite{dfw}, provides a unifying perspective on a number of known results
relating geometric objects intrinsically tied to the volume form of the ambient
space with its multiply-foliated geometry in question, such as the following.
\begin{prop}[\cite{vaisman} eq. (3.17)]
  \label{thm:dfw-ricci}
  Let $(M,\omega,\fol,\gol)$ be a bi-Lagrangian manifold of dimension $2n$ and
  let $\nabla$ be its canonical connection (see Definition
  $\ref{def:s2w-lagr-conn}$). The expression for the Ricci tensor of $\nabla$
  in coordinates $(x_1,\ldots,x_n,$ $y_1,\ldots,y_n)$ satisfying $T\fol =
  \bigcap_{i=1}^n \ker dy_i$ and $T\gol = \bigcap_{j=1}^n \ker dx_j$ is exactly
  \begin{equation}
    \Rc = \frac{\partial^2\log\abs{\det A(x,y)}}{\partial x_i\partial
    y_j}\,dx_idy_j,
  \end{equation}
  where $\maps{A}{\Rb{2n}}{\mathrm{GL}(2n,\Rb{1})}$ is a matrix-valued function
  with entries $A_{ij}$ given by
  \begin{equation}
    \label{eq:s2w-lagr-ricci}
    \omega = \sum_{i,j=1}^nA_{ij}(x,y)\,dx_i\wedge dy_j.
  \end{equation}
\end{prop}

The first part of this work is motivated by another way of relating
symplectic geometry to unimodular geometry. It stems from the fact that in
dimension $2$ the notions of symplectic form and volume form coincide, so that
planar bi-Lagrangian structures are exactly the planar divergence-free
$2$-webs. Combined with a general observation that the essence of symplectic
structure is encoded in its $2$-dimensional objects, such as the symplectic
form itself or pseudoholomorphic curves \cite{mcduff}, it leads naturally to
the following research strategy: instead of inflating a $2$-dimensional
symplectic form $\omega$ to a top-dimensional form by taking its $n^\text{th}$
power, restrict $\omega$ to various $2$-dimensional submanifolds, on which the
symplectic form itself becomes a~volume form. If the submanifolds are generic
enough, they inherit the bi-Lagrangian structure from the ambient bi-Lagrangian
manifold, which allows us to draw conclusions about them using the arsenal of
tools and intuitions developed for webs in unimodular geometry (which we list
in Section $\ref{ch:s2w-lagr-dfw}$ for convenience of the reader) in order to
translate them back to novel results regarding the ambient structure.

This broad strategy brings us to the main result of the first part of
the paper, Theorem $\ref{thm:s2w-lagr-equiv}$, which characterizes flatness of
the canonical connection of the bi-Lagrangian manifold in terms of flatness of
all of the canonical connections of its $2$-dimensional bi-Lagrangian
submanifolds $S$ of a certain class, which is in turn equivalent to vanishing
of a unimodular-geometric invariant of these submanifolds called the
\emph{volume-preserving reflection holonomy} of a codimension-$1$ web equipped
with a volume form. In a very concrete sense, this invariant measures the
deviation of the curvature of the canonical connection of $S$ from $0$ (see
Lemma $\ref{thm:dfw-loop-taylor}$ and the preceding paragraph).

The proof of this theorem is the main focus of Section
$\ref{ch:s2w-lagr-flat}$. It uses in an essential way the results of I. Vaisman
\cite{vaisman-fol, vaisman} regarding the \emph{symplectic curvature tensor}
$Rs$ of a symplectic connection $\nabla$ with curvature endomorphism $R$, which
is given by the formula
\begin{equation}
  Rs(X,Y,Z,W) = \omega(R(Z,W)Y,X)\qquad\text{for } X,Y,Z,W \in\mathfrak{X}(M),
\end{equation}
together with a known correspondence between bi-Lagrangian geometry and
\emph{para-Kähler geometry} \cite{etayosantamaria,bilagrangian,hess},
by way of which one assigns to each bi-Lagrangian structure
$(M,\omega,\fol,\gol)$ of dimension $2n$ a uniquely determined \emph{neutral
metric $g$} of signature $(n,n)$ equipped with an integrable para-complex
structure $J$ compatible with $g$, which additionally satisfies $\nabla J = 0$
with respect to the Levi-Civita connection $\nabla$ of $g$ \cite{paracomplex}.
This correspondence allows us to relate the curvature of the bi-Lagrangian
connection of the ambient space with the curvature of its bi-Lagrangian
submanifolds using the classical tools of pseudo-Riemannian geometry, such as
bi-Lagrangian analogues of the \emph{Gauss equation}
$(\ref{eq:s2w-lagr-gauss})$, second fundamental form and sectional curvature.

All of the above ingedients, mixed with the characterization of the matrix $A$
of the symplectic form $\omega$ of a bi-Lagrangian structure
$(M,\omega,\fol,\gol)$ with a flat canonical connection $\nabla$ (Theorem 
$\ref{thm:s2w-lagr-fg}$) obtained using the formula for curvature $2$-forms of
$\nabla$ via Cartan's method of moving frames, lead jointly to the
geometric characterization of flat bi-Lagrangian structures.

\medskip

In the second part of this paper we present our approach to two
interesting questions posed by S. Tabachnikov in \cite{2-webs}[{}I.6]. The
subjects of both questions are bi-Lagrangian structures constructed from a pair
of some immersed submanifolds $L,K$, the properties of which reflect in some
way the mutual arrangement of $L,K$ inside the ambient space. It is natural to
expect that flatness of the induced bi-Lagrangian structures impacts in some
way the shape of the corresponding pairs of submanifolds. The problem in both
cases is to find and characterize such pairs: \emph{for which pairs $L,K$ the
induced bi-Lagrangian structure is flat?}

The two questions involve the following classes of structures:
\begin{enumerate}[label=$(\arabic{enumi})$, ref=\arabic{enumi}]
  \item\label{it:s2w-problem-rays}
    bi-Lagrangian structure-germs on the space of oriented rays $T^*S^{n-1}$
    induced by a pair of generic hypersurface-germs $L,K$, the construction of
    which was given at the beginning of this section, and
  \item\label{it:s2w-problem-tangents}
    bi-Lagrangian structure-germs on $\Rb{2n}$ equipped with the standard
    symplectic form $\omega = \sum_{i=1}^n dp_i\wedge dq_i$, the foliations of
    which are composed of (Lagrangian) affine tangent spaces to a pair of
    generic immersed Lagrangian submanifolds.
\end{enumerate}

We provide a full solution to problem $(\ref{it:s2w-problem-rays})$ and a
solution for the case $n=1$ of problem $(\ref{it:s2w-problem-tangents})$. The
former is summarized in the statement of the following theorem, the proof of
which follows directly from Theorem $\ref{thm:s2w-lagr-rays-main}$, Theorem
$\ref{thm:s2w-lagr-rays-main-n3}$ and the surrounding remarks.

\begin{thm}
  \label{thm:s2w-lagr-rays-flat}
  Let $(U,\omega,\fol,\gol)$ be a bi-Lagrangian structure on the open subset
  $U$ of space of rays $T^*S^{n-1}$ induced by a pair of hypersurfaces
  $H,K\subseteq\Rb{n}$ and let $\nabla$ be its canonical connection. Then
  $\nabla$ is never flat. Moreover,
  \begin{enumerate}[label=$(\alph{enumi})$, ref=\alph{enumi}]
    \item if $n\neq 3$, then $\nabla$ is not Ricci-flat.
    \item if $n=3$, then $\nabla$ is Ricci-flat if and only if $H$ and $K$ are
      disjoint open subsets $V_H,V_K\subseteq S_{c,r}^2$ of a single $2$-sphere
      $S_{c,r}^2\subseteq\Rb{3}$ of arbitrary positive radius $r>0$ with center
      at any point $c\in\Rb{3}$.
  \end{enumerate}
\end{thm}

As for the second problem, under a mild regularity assumption it is also the
case that there are no curves $L,K$ such that the canonical connection $\nabla$
of the bi-Lagrangian structure induced by tangents to $L,K$ is flat. We refer
to Theorem $\ref{thm:s2w-lagr-tangents-2d}$ for a precise statement and the
proof.

In both cases, our proofs reduce to raw calculations involving equalities
expressing the vanishing of the curvature of $\nabla$. The complexity of these
systems of equalities would render them very difficult to solve if not for the
fact, that, given a parametrization of points of the bi-Lagrangian manifold in
question by pairs $(x(s),y(t))\in L\times K$ with $s,t\in\Rb{k}$, the
expressions involved are rational functions of the derivatives of $x(s),y(t)$
of up to fourth order. Since these are handled well by a computer algebra
system, it opens up the possibility of giving a~(mostly) algebraic proof of the
above results, which we pursue and complete in Sections
$\ref{ch:s2w-lagr-rays}$ and $\ref{ch:s2w-lagr-tangents}$.

\bigskip\noindent\textbf{\emph{Acknowledgements}}

\smallskip We would like to thank the anonymous referees for their helpful
comments and suggestions regarding our manuscript, and for sharing with us some
very interesting references on various topics in (para-)Kähler geometry.

\section{Preliminaries on webs in unimodular geometry}
\label{ch:s2w-lagr-dfw}

A \emph{divergence-free $n$-web} is a structure
$\web_\Omega=(M,\Omega,\fol_1,\ldots,\fol_n)$ consisting of a smooth manifold
$M$, a volume form $\Omega$ on $M$ and a collection of $n$ foliations
$\fol_1,\ldots,\fol_n$ of $M$ generated by tangent distributions
$T\fol_1,\ldots,T\fol_n$ which are in general position. In our case, the
general position assumption asserts that the equality $\sum_{i=1}^n \codim
T\fol_i = \codim \cap_{i=1}^n T\fol_i$ holds. The divergence-free $n$-web
$\web_\Omega$ is said to be \emph{of codimension $1$} if each folation of
$\web_\Omega$ has codimension $1$. In particular, a planar codimension-$1$
divergence-free web is the same as a planar bi-Lagrangian structure.

Let us denote by $\Gamma(T\fol_i)$ the $C^\infty(M)$-module of vector fields
tangent to leaves of $\fol_i$, and by $\Omega^k(M;T\fol_i)$ the
$C^\infty(M)$-module of differential $k$-forms with values in $T\fol_i$. For an
affine connection $\nabla$ and $X\in\Gamma(T\fol_i)$, the image of $X\mapsto
\nabla X$ belongs to $\Omega^1(M;TM)$, since $v\mapsto \nabla_vX$ is fiberwise
linear in $v$. Thus, we can write
$\nabla\Gamma(T\fol_i)\subseteq\Omega^1(M;TM)$. To each divergence-free $n$-web
of codimension-$1$ one can associate a canonical affine connection which is
defined by several desirable properties.

\begin{prop}[\cite{dfw}]
  \label{thm:dfw-conn}
  Let $\web_\Omega = (M, \Omega, \fol_1, \ldots, \fol_n)$ be a codimension-$1$
  divergence-free $n$-web. There exists a unique torsionless connection
  $\nabla$, called the \emph{$\web_\Omega$-connection}, which satisfies
  \begin{enumerate}[label=$(\arabic{enumi})$, ref=\arabic{enumi}]
    \item \label{thm:dfw-conn:fols}
      $\nabla\Gamma(T\fol_i)\subseteq\Omega^1(M;T\fol_i)$ for each
        $i=1,\ldots,n$,\hfill
          ($\fol_i$ are $\nabla$-parallel)
    \item \label{thm:dfw-conn:vol}
      $\nabla\Omega=0$. \hfill
        (the volume form $\Omega$ is $\nabla$-parallel)
  \end{enumerate}
\end{prop}

The curvature of the $\web_\Omega$-connection $\nabla$ measures the
non-triviality of a certain geometric invariant called \emph{volume-preserving
reflection-holonomy at $p\in M$} of $\web_\Omega$ \cite{dfw}. This invariant
was first described in a side remark by Tabachnikov in \cite[268]{2-webs} in
the special case of planar webs. It arises as a group of all
diffeomorphism-germs $\ell_{p;\fol_i,\fol_j}$ with $i,j=1,\ldots,n$ defined in
the following way.

Each point $q$ near $p$ determines a collection of $2n$ leaves of the
foliations of $\web_\Omega$ which collectively bound a certain compact region
$[p,q]$ in the shape of a coordinate cube in some coordinate system adapted
to the web $\web_\Omega$. These \emph{regions bounded by leaves of
$\web_\Omega$} are said to be \emph{adjacent along $F\in\fol_k$} if they share
a side which lies in its entirety inside the leaf $F$. In particular, two
regions $K,L$ bounded by leaves of $\web_\Omega$ adjacent along $F\in\fol_k$
form a larger region $K\cup L$ whenever both $K$ and $L$ are compact subsets of
a single $\web_\Omega$-adapted coordinate domain. In this case, we say that $F$
\emph{subdivides $K\cup L$ into subregions $K$ and $L$}. Given a point $q\in
M$, let $o=r_{p;\fol_k}(q)$ be a point different from $q$ defining a
region $[p,o]$ bounded by leaves of $\web_\Omega$ which is adjacent to $[p,q]$
along the leaf of $\fol_k$ crossing $p$ and such that the two regions have
equal volumes with respect to the volume form $\Omega$. This relation between
$q$ and $o$ extends to a unique smooth map-germ
$\maps{r_{p;\fol_k}}{(M,p)}{(M,p)}$. The generators of the volume-preserving
reflection-holonomy group at $p\in M$ are exactly the diffeomorphism-germs
$\ell_{p;\fol_i,\fol_j} = r_{p;\fol_j}\circ r_{p;\fol_i}\circ r_{p;\fol_j}\circ
r_{p;\fol_i}$ for $i\neq j$, $i,j=1,\ldots,n$.

\begin{lem}
  \label{thm:dfw-loop-taylor}
  Let $\web_\Omega=(M,\fol_1,\ldots,\fol_n,\Omega)$ be a divergence-free
  $n$-web of codimension $1$. Fix a~point $p\in M$ and a $\web_\Omega$-adapted
  coordinate system $(x_1,\ldots,x_n)$ centered at $p$. Express the volume form
  as $\Omega=h(x)\,dx_1\wedge dx_2\wedge\cdots\wedge dx_n$ and the
  Ricci tensor of the $\web_\Omega$-connection $\nabla$ at~$p\in M$ as
  $\operatorname{Rc}_{|p} = \sum_{i\neq j} \kappa_{ij}\,dx_idx_j$. In this
  setting, the volume-preserving loop along the foliations $\fol_i$, $\fol_j$
  with $T\fol_i=\ker dx_i$ and $T\fol_j=\ker dx_j$ satisfies
  \begin{equation}
    \ell_{p;\fol_i,\fol_j}(x) =
      (x_1,\ldots,x_{i-1},u_i(x),
        x_{i+1},\ldots,x_{j-1},u_j(x),x_{j+1},\ldots,x_n)
  \end{equation}
  where the $i^{\text{th}}$ and $j^{\text{th}}$ coordinates of the image
  satisfy
  \begin{equation}
    \begin{aligned}
      u_i(x) &= x_i + 2\kappa_{ij} x^2_ix_j + o(|x|^3)\quad\text{and} \\
      u_j(x) &= x_j - 2\kappa_{ij} x_ix^2_j + o(|x|^3).
    \end{aligned}
  \end{equation}
  respectively.
\end{lem}

The result below gives several geometric conditions for local triviality of a
codimension-$1$ divergence-free $n$-web $\web_\Omega$, one of which directly
involves the reflection-holonomy of $\web_\Omega$. We will use it to
characterize locally trivial bi-Lagrangian structures in terms of their
particular two-dimensional substructures in Theorem \ref{thm:s2w-lagr-equiv}.

\begin{thm}[\cite{dfw}]
  \label{thm:dfw-geom}
  Let $\web_\Omega=(M,\Omega,\fol_1,\ldots,\fol_n)$ be a codimension-$1$
  divergence-free $n$-web. The following conditions are equivalent.
  \begin{enumerate}[label=$(\arabic*)$, ref=(\arabic*)]
    \item\label{thm:dfw-geom:triv}
      The divergence-free web $\web_\Omega$ is locally trivial, meaning that
      the space $M$ can be covered with coordinate charts $(x_1,\ldots,x_n)$ in
      which $T\fol_i = \ker dx_i$ for $i=1,\ldots,n$ and $\Omega =
      dx_1\wedge\cdots\wedge dx_n$.
    \item\label{thm:dfw-geom:taba}
      For each pair $\fol,\gol\in\Fols(\web_\Omega)$ of two different
      foliations of $M$, any region bounded by leaves $K$, and any two open
      subsets of leaves $F\in\fol$, $G\in\gol$ which subdivide $K$ into four
      subregions $A,B,C,D$ with $(A\cup B)\cap(C\cup D)\subseteq F$ and $(A\cup
      D)\cap(B\cup C)\subseteq G$, the respective $\Omega$-volumes $a,b,c,d$ of
      $A,B,C,D$ satisfy
      \begin{equation}
        ac=bd.
      \end{equation}
    \item\label{thm:dfw-geom:cut}
      For each pair $\fol,\gol\in\Fols(\web_\Omega)$ of two different
      foliations of $M$, any region bounded by leaves $K$, and any two open
      subsets of leaves $F\in\fol$, $G\in\gol$ which subdivide $K$ into four
      subregions $A,B,C,D$ with $(A\cup B)\cap(C\cup D)\subseteq F$ and $(A\cup
      D)\cap(B\cup C)\subseteq G$ in such a way that the $\Omega$-volumes
      $a,b,c,d$ of $A,B,C,D$ satisfy $a+b = c+d$, the equality $a=b$ implies
      $a=b=c=d$.
    \item\label{thm:dfw-geom:split}
      For any region bounded by leaves $K$ and each $k=1,2,\ldots,n$ there
      exist open subsets of leaves $F_i\in\fol_i$ for $i=1,2,\ldots,k$ which
      subdivide $K$ into $2^k$ subregions with equal $\Omega$-volumes.
    \item\label{thm:dfw-geom:hol}
      The volume-preserving reflection holonomy of $\web_\Omega$ at each
      point $p\in M$ is trivial.
    \item\label{thm:dfw-geom:ricci}
      The $\web_\Omega$-connection $\nabla$ is Ricci-flat.
    \item\label{thm:dfw-geom:flat}
      The $\web_\Omega$-connection $\nabla$ is flat.
  \end{enumerate}
\end{thm}

The proofs of the aforementioned theorems can be found in \cite{dfw} and, since
they are quite lengthy, they will not be given here. We refer the interested
reader to the original paper.

\section{Flatness of bi-Lagrangian structures}
\label{ch:s2w-lagr-flat}

\subsection{Bi-Lagrangian connection}

Let $(M,\omega)$ be a symplectic manifold of dimension $2n$. A \emph{Lagrangian
foliation} of $(M,\omega)$ is a foliation $\fol$ such that each leaf $L\in\fol$
is a Lagrangian submanifold of $(M,\omega)$, meaning that $\dim L = n$ and
$\omega_{|L} = 0$. A quadruple $\web_\omega=(M,\omega,\fol,\gol)$ consisting of
$(M,\omega)$ and two Lagrangian foliations $\fol$, $\gol$ of $(M,\omega)$, the
leaves of which intersect transversely at each point $p\in M$, is called a
\emph{bi-Lagrangian manifold} \cite{bilagrangian} or a \emph{Lagrangian
$2$-web} \cite{2-webs}.

Each such structure carries a unique symplectic connection $\nabla$ which
parallelizes both of its foliations.
\begin{defn}[\cite{bilagrangian,vaisman-fol}]
  \label{def:s2w-lagr-conn}
  Let $\web_\omega=(M,\omega,\fol,\gol)$ be a bi-Lagrangian manifold. A
  connection $\nabla$ is said to be a \emph{bi-Lagrangian
  $\web_\omega$-connection} (or a \emph{canonical connection of a bi-Lagrangian
  manifold \cite{vaisman}}) if the following conditions hold:
  \begin{enumerate}[label=$(\alph{enumi})$, ref=\alph{enumi}]
    \item\label{def:s2w-lagr-conn:sympl}
      $\nabla$ is \emph{almost symplectic}, i.e. $\nabla_v\omega = 0$ for each
      $v\in TM$ \cite{vaisman},
    \item\label{def:s2w-lagr-conn:fol}
      $\nabla_v \Gamma(T\fol) \subseteq T\fol$ and $\nabla_v \Gamma(T\gol)
      \subseteq T\gol$ for each $v\in TM$,
    \item\label{def:s2w-lagr-conn:tor}
      $\nabla_XY-\nabla_YX=[X,Y]$ for each $X\in\Gamma(T\fol)$ and
      $Y\in\Gamma(T\gol)$.
  \end{enumerate}
\end{defn}

The action of a bi-Lagrangian connection $\nabla$ on $TM$ can be fully
recovered from the above definition using certain natural maps associated to
$\web_\omega$, which we define below.
Since the leaves of $\fol$ and $\gol$ are
transverse to each other, the tangent bundle decomposes into a Whitney sum $TM
= T\fol\oplus T\gol$. We denote the corresponding bundle projections by
\begin{equation}
  \label{eq:s2w-lagr-proj}
  \maps{\pi_\fol}{TM}{T\fol};\ v\mapsto v_\fol,\qquad
  \maps{\pi_\gol}{TM}{T\gol};\ v\mapsto v_\gol,
\end{equation}
where $v=v_\fol+v_\gol$ and $v_\fol\in T\fol$, $v_\gol\in T\gol$. This
decomposition allows us to identify the normal bundle $\nu\fol = TM/T\fol$ with
$T\gol$ and $\nu\gol=TM/T\gol$ with $T\fol$ in a natural way.
Additionally, the restricitons of the usual contraction isomorphism $TM\simeq
T^*M; v\mapsto\iota_v\omega$ to $T\fol$ and $T\gol$ descend to isomorphisms
\begin{equation}
  \label{eq:s2w-lagr-isos}
  \maps{\alpha}{T\fol}{(TM/T\fol)^*}\simeq T^*\gol,\qquad
  \maps{\beta}{T\gol}{(TM/T\gol)^*}\simeq T^*\fol,
\end{equation}
since $T\fol$ and $T\gol$ are Lagrangian subbundles of $TM$. These two maps are
bound be the duality relation $\alpha = -\beta^*$, where $\beta^*$ is the
transpose of $\beta$.

\begin{prop}[\cite{hess,vaisman-fol}]
  \label{thm:s2w-lagr-formula}
  Let $\web_\omega=(M,\omega,\fol,\gol)$ be a bi-Lagrangian manifold and let
  the maps $\alpha,\beta,\pi_\fol,\pi_\gol$ be given by
  $(\ref{eq:s2w-lagr-proj})$ and $(\ref{eq:s2w-lagr-isos})$. The action of a
  bi-Lagrangian $\web_\omega$-connection $\nabla$ on $TM$ is given by the
  unique $\Rb{1}$-linear extension of
  \begin{enumerate}[label=$(\alph{enumi})$, ref=\alph{enumi}]
    \item\label{thm:s2w-lagr-formula:bott-fol}
      \makebox[10.2em][l]{$\nabla_{X_\fol}Y_\gol = \pi_\gol[X_\fol, Y_\gol]$} for
      $X_\fol\in\Gamma(T\fol)$ and
      $Y_\gol\in\Gamma(T\gol)$,
    \item\label{thm:s2w-lagr-formula:bott-gol}
      \makebox[10.2em][l]{$\nabla_{X_\gol}Y_\fol = \pi_\fol[X_\gol, Y_\fol]$} for
      $X_\gol\in\Gamma(T\gol)$ and
      $Y_\fol\in\Gamma(T\fol)$,
    \item\label{thm:s2w-lagr-formula:alpha}
      \makebox[10.2em][l]{$\nabla_{X_\fol}Y_\fol
        = \alpha^{-1}\nabla_{X_\fol}\alpha Y_\fol$} for
      $X_\fol\in\Gamma(T\fol)$ and
      $Y_\fol\in\Gamma(T\fol)$,
    \item\label{thm:s2w-lagr-formula:beta}
      \makebox[10.2em][l]{$\nabla_{X_\gol}Y_\gol
        = \beta^{-1}\nabla_{X_\gol}\beta Y_\gol$} for
      $X_\gol\in\Gamma(T\gol)$ and
      $Y_\gol\in\Gamma(T\gol)$.
  \end{enumerate}
\end{prop}
\begin{proof}
  To obtain $(\ref{thm:s2w-lagr-formula:bott-fol})$ and
  $(\ref{thm:s2w-lagr-formula:bott-gol})$, apply the projections $\pi_\fol$ and
  $\pi_\gol$ to property $(\ref{def:s2w-lagr-conn:tor})$ of Definition
  \ref{def:s2w-lagr-conn}. For $(\ref{thm:s2w-lagr-formula:alpha})$ and
  $(\ref{thm:s2w-lagr-formula:beta})$, expand the left-hand side of the
  equality $\nabla_{v}\omega=0$. For $v=X_\fol$ and $w\in T\gol$ we
  obtain
  \[
    \begin{aligned}
      \alpha\,\nabla_{X_\fol}Y_\fol(w) &= \omega(\nabla_{X_\fol}Y_\fol,w) \\
        &= X_\fol\omega(Y_\fol,w) - \omega(Y_\fol,\nabla_{X_\fol}w) \\
        &= \nabla_{X_\fol}(\iota_{Y_\fol}\omega)(w) = \nabla_{X_\fol}\alpha
          Y_\fol(w),
    \end{aligned}
  \]
  proving $(\ref{thm:s2w-lagr-formula:alpha})$;
  $(\ref{thm:s2w-lagr-formula:beta})$ is proven identically.
\end{proof}

The connection $\nabla$ given in Proposition \ref{thm:s2w-lagr-formula} is
well-defined and unique. Its action on $TM/TL\simeq T\gol_{|L}$ along the
leaf $L$ of $\fol$ given by $(\ref{thm:s2w-lagr-formula:bott-fol})$ of
Proposition \ref{thm:s2w-lagr-formula} is exactly the (conjugate by the
projection $\pi_\gol$ of the) action of \emph{Bott's connection}
$\maps{D^\fol}{\Gamma(T\fol)\times\Gamma(TM/T\fol)}{TM/T\fol}$ restricted to
$L$ \cite{hess}. The connection $\nabla$ is torsionless \cite[Proposition
3.1]{vaisman}. This can be seen by considering $(\ref{def:s2w-lagr-conn:fol})$
of Definition \ref{def:s2w-lagr-conn} and the identity
\begin{equation}
  \label{eq:s2w-lagr-tor}
    d\omega(X,Y,Z) = \omega(T(X,Y),Z) + \omega(T(Y,Z),X) + \omega(T(Z,X),Y)
\end{equation}
involving the torsion tensor $T$ of $\nabla$, which is valid for any
$X,Y,Z\in\mathfrak{X}(M)$ and any connection $\nabla$ satisfying
$\nabla_v\,\omega=0$ for each $v\in TM$. The above remark about torsion,
Definition \ref{def:s2w-lagr-conn} and Proposition \ref{thm:s2w-lagr-formula}
characterize $\nabla$ as the unique symplectic connection in the sense of
\cite{fedosov} extending both Bott's connections associated with the foliations
of the web $\web_\omega$. Bott's connections are flat along the leaves of their
respective foliations; that the same holds for the bi-Lagrangian connection
$\nabla$ is a direct consequence of the following lemma.

\begin{lem}
  \label{thm:s2w-lagr-frames}
  Let $\web_\omega=(M,\omega,\fol,\gol)$ be a bi-Lagrangian manifold of
  dimension $2n$ and let $(x_1,\ldots,x_n,y_1,\ldots,y_n)$ be a local
  coordinate system on $M$ in which $T\fol = \bigcap_{i=1}^n \ker dy_i$ and
  $T\gol = \bigcap_{j=1}^n\ker dx_j$ for $i,j=1,\ldots,n$. Also, let $L$ be a
  leaf of $\fol$ and $K$ be a leaf of $\gol$. Then
  \begin{enumerate}[label=$(\alph{enumi})$, ref=\alph{enumi}]
    \item\label{thm:s2w-lagr-frames:f}
      the vector fields $(X_{y_i}, \basis{y_j})_{i,j=1}^n$
      form a local $\nabla$-parallel frame along $L$,
    \item\label{thm:s2w-lagr-frames:g}
      the vector fields $(\basis{x_i}, X_{x_j})_{i,j=1}^n$
      form a local $\nabla$-parallel frame along $K$,
  \end{enumerate}
  where $X_f$ denotes the Hamiltonian vector field corresponding to a smooth
  function $f\in C^\infty(M)$ defined by $\omega(X_f,\cdot)=df$.
\end{lem}
\begin{proof}
  Consider $(\ref{thm:s2w-lagr-frames:f})$ without loss of generality and fix
  $Y\in\Gamma(T\fol)$. Note that $\nabla_{Y}\basis{y_i}=0$ by
  $(\ref{thm:s2w-lagr-formula:bott-fol})$ of Proposition
  \ref{thm:s2w-lagr-formula} for each given $i=1,2,\ldots,n$. Now, using the
  fact that $\nabla_YX\in\Gamma(T\fol)$ for every $X\in\Gamma(T\fol)$ by
  Definition \ref{def:s2w-lagr-conn}, we obtain that for each smooth function
  $H$ which is constant on leaves of $\fol$ the equality
  \begin{equation}
    \label{eq:s2w-lagr-leafwise-const}
    \begin{aligned}
      (\nabla_Y\,dH)(X)
        &= Y\,\big(dH(X)\big) - dH(\nabla_Y X)
         = Y\,\big(dH(X_\gol)\big) - dH(\nabla_YX_\fol) - dH([Y,X_\gol]_\gol) \\
        &= Y\,\big(dH(X_\gol)) - dH([Y,X_\gol])
         = d(dH)(Y,X_\gol) + X_\gol(YH) = 0
    \end{aligned}
  \end{equation}
  holds, where for each vector field $V\in\mathfrak{X}(M)$ we denote by
  $V_\fol\in\Gamma(T\fol)$ and $V_\gol\in\Gamma(T\gol)$ the projections of $V$
  onto $T\fol$ and $T\gol$. The coordinate $y_i$ is such a function, hence the
  claim follows from $(\ref{thm:s2w-lagr-formula:alpha})$ of Proposition
  \ref{thm:s2w-lagr-formula}, where $dy_i = \alpha X_{y_i}$ by definition.
\end{proof}

The existence of $\nabla$-parallel frames along the leaves of $\web_\omega$
allows us to conclude that the Riemann curvature endomorphism $R(X,Y)$ of
$\nabla$ is zero for pairs of vector fields $X,Y\in\Gamma(T\fol)$ and
$X,Y\in\Gamma(T\gol)$. This fact is reflected in the expression for the
curvature of $\nabla$ in a local coordinate system
$(x_1,\ldots,x_n,y_1,\ldots,y_n)$ compatible with $\web_\omega$ given below. We
state it using Cartan's moving frame formalism (see e.g.
\cite{chernchen,foundg1}) and a certain natural decomposition of the exterior
derivative operator $d=d_x+d_y$ associated to $\web_\omega$. The operators
$d_x,d_y$ are exactly the skew-derivations on $\Omega^*(M)$ satisfying
$d_yd_x+d_xd_y=0$ and $d_x^2=d_y^2=0$, together with $d_x f(v) = v_\fol f$ and $d_y
f(v) = v_\gol f$ for any $f\in C^\infty(M)$, $v\in TM$ and $v_\fol\in T\fol,
v_\gol\in T\gol$ bound by the equality $v=v_\fol+v_\gol$ \cite{chainderham}.

\begin{prop}[\cite{vaisman}]
  \label{thm:s2w-lagr-conn-coords}
  Let $\web_\omega=(M,\omega,\fol,\gol)$ be a bi-Lagrangian manifold of
  dimension $2n$ with canonical connection $\nabla$ and let
  $(x_1,\ldots,x_n,y_1,\ldots,y_n)$ be a local coordinate system in which
  $T\fol=\bigcap_{i=1}^n\ker dy_i$ and $T\gol=\bigcap_{i=1}^n\ker dx_i$. Let
  $A$ be the $n\times n$ matrix with entries $A_{ij} =
  \omega(\basis{x_i},\basis{y_j})$. The matrix of connection $1$-forms $\gamma$
  of $\nabla$ with respect to the coordinate frame is
  \begin{equation}
    \gamma = \bm{ (d_xA\cdot A^{-1})^T & 0 \\ 0 & (A^{-1}\cdot d_yA)},
  \end{equation}
  while the matrix of curvature $2$-forms is
  \begin{equation}
    \label{eq:s2w-lagr-conn-coords-curv}
    \Omega = \bm{ d_y(d_xA\cdot A^{-1})^T & 0 \\ 0 & d_x(A^{-1}\cdot d_yA)}.
  \end{equation}
\end{prop}
\begin{proof}
  The off-diagonal terms are zero since $\nabla$ preserves $T\fol$ and
  $T\gol$ by property $(\ref{def:s2w-lagr-conn:fol})$ of Definition
  \ref{def:s2w-lagr-conn}. To find the diagonal terms, use Lemma
  \ref{thm:s2w-lagr-frames} to differentiate $\basis{x_i}$ and $\basis{y_j}$
  along $X=X_\fol+X_\gol\in\mathfrak{X}(M)$ with $X_\fol\in\Gamma(T\fol)$,
  $X_\gol\in\Gamma(T\gol)$. This yields
  \begin{equation}
    \begin{aligned}
      &\nabla_X\basis{x_i}
        = \nabla_{X_\fol}\basis{x_i} + \nabla_{X_\gol}\basis{x_i}
        = \nabla_{X_\fol}\basis{x_i}
        = \alpha^{-1}\nabla_{X_\fol}\alpha\basis{x_i} \\
        &\hskip 1em = \smsum{j}\alpha^{-1}\nabla_{X_\fol}A_{ij}dy_j
        = \smsum{j}(\alpha^{-1}dy_j)(X_\fol A_{ij}) \\
        &\hskip 1em = \smsum{j,k}(\basis{x_k}A^{-1}_{jk})(X_\fol A_{ij})
        = \smsum{k}\basis{x_k}\,(d_xA(X)\cdot A^{-1})^T_{ki},
    \end{aligned}
  \end{equation}
  which proves that the upper-left block is $(d_xA(X)\cdot A^{-1})^T$. Bearing
  in mind that $\beta\basis{y_i} = \smsum{j}(-A_{ji})dx_j$, one similarly
  obtains the lower-right entries of $\gamma$. To determine the curvature
  forms, use Cartan's structure equation \cite[Chapter 2, \S 5]{foundg1} to
  arrive at
  \begin{equation}
    \Omega = d\gamma + \gamma\wedge\gamma
       = \bm{\Omega_\fol & 0 \\ 0 & \Omega_\gol},
  \end{equation}
  where
  \begin{equation}
    \begin{aligned}
      \Omega_\fol &= d(A^{-T}\cdot d_xA^T) + (A^{-T}\cdot
        d_xA^T)\wedge(A^{-T}\cdot d_xA^T) \text{ and}\\
      \Omega_\gol &= d(A^{-1}\cdot d_yA)
        + (A^{-1}\cdot d_yA)\wedge(A^{-1}\cdot d_yA).
    \end{aligned}
  \end{equation}
  Note that
  \begin{equation}
    \begin{aligned}
      &d_y(A^{-1}\cdot d_yA) = d_y(A^{-1})\wedge d_yA
        + A^{-1}\cdot \smash{\overbrace{d_y^2A}^{=0}} \\
        &\hskip 2em= (-A^{-1}\cdot d_yA\cdot A^{-1})\wedge d_yA \\
        &\hskip 2em= -(A^{-1}\cdot d_yA)\wedge(A^{-1}\cdot d_yA),
    \end{aligned}
  \end{equation}
  so $\Omega_\gol$ reduces to $d_x(A^{-1}\cdot d_yA)$. The other term is
  handled analogously.
\end{proof}

The two operators $d_x$, $d_y$ are part of a certain double complex
$\Omega^{\bullet,\bullet}(M, \web_\omega)$. The spaces
$\Omega^{p,q}(M,\web_\omega)$ consist of $(p+q)$-forms which annihilate every
wedge product of $k$ vectors tangent to $\fol$ and $l$ vectors tangent to
$\gol$ with $k+l=p+q$, $k\neq p$ and $l\neq q$. Since $T\fol$, $T\gol$ are
involutive, the derivations $d_x$, $d_y$ have degree $(1,0)$, $(0,1)$ in
$\Omega^{\bullet,\bullet}(M,\web_\omega)$ respectively \cite{chainderham}.
Moreover, the following variant of local Poincaré's lemma holds
\cite[$(15)$]{chainderham}.
Assume that $M$ is contractible along $T\gol$ to a leaf of $\fol$. If
$\alpha\in\Omega^{p+1,q}(M,\web_\omega)$ satisfies $d_x\alpha = 0$ for some
$p,q\in\mathbb{N}$, then there exist $\beta\in\Omega^{p,q}(M,\web_\omega)$ such
that $\alpha = d_x\beta$. Moreover, in the case when $M=U\times V$ with $TU =
T\fol$ and $TV = T\gol$, if $\alpha\in\Omega^{0,q}(M,\web_\omega)$ satisfies
$d_x\alpha = 0$, then $\alpha = \pi_V^*\beta$ for some $\beta\in\Omega^q(V)$,
where $\maps{\pi_V}{U\times V}{V}$ is a projection onto the second factor. An
analogous theorem is true for the other operator $d_y$.

These tools allow us to express flatness of bi-Lagrangian structures in terms
of the matrix $A$ of $\omega$ inside any local coordinate system adapted to
$\web_\omega$.

\begin{thm}
  \label{thm:s2w-lagr-fg}
  In the context of Proposition \ref{thm:s2w-lagr-conn-coords}, the
  bi-Lagrangian connection $\nabla$ of $\web_\omega$ is flat if and only if
  there exist matrix-valued function-germs $\maps{f,g}{(\Rb{n},0)}{M_{n\times
  n}(\Rb{1})}$ which satisfy
  \begin{equation}
    \label{eq:s2w-lagr-afg}
    A(x,y) = f(x)\cdot g(y).
  \end{equation}
\end{thm}
\begin{proof}
  Assume that $\nabla$ is flat. Then, according to
  $(\ref{eq:s2w-lagr-conn-coords-curv})$, we have $d_y(d_xA\cdot A^{-1})=0$. By
  a variant of Poincaré's lemma outlined in the remark above, $d_xA\cdot A^{-1}
  = \beta(x)$ for some matrix-valued $1$-form-germ $\beta$ with
  $\beta_{ij}\in\Omega^{1,0}(M,\web_\omega)$.

  Consider a system of differential equations $d_xg = \beta(x)g$ on $\Rb{n}$ for an
  $n\times n$ matrix $g$. This system has at least one invertible solution,
  say, $A(x,0)$. Let $f(x)$ be one of them. The equality $d_x A\cdot A^{-1} =
  d_x f(x)\cdot f(x)^{-1}$ implies
  \begin{equation}
    \begin{aligned}
      d_x (f(x)^{-1}\cdot A) &= (-f(x)^{-1} d_xf(x)f(x)^{-1})\cdot A +
        f(x)^{-1}\cdot d_x A \\
      &= -f(x)^{-1}\cdot d_x A \cdot (A^{-1} A) + f(x)^{-1}\cdot d_xA = 0.
    \end{aligned}
  \end{equation}
  Hence, again by Poincaré's lemma, there exists a matrix-valued function-germ
  $g(y)$ such that $f(x)^{-1}\cdot A = g(y)$. This is equivalent to
  $(\ref{eq:s2w-lagr-afg})$.

  If $(\ref{eq:s2w-lagr-afg})$ holds, then a straightforward calculation of
  $(\ref{eq:s2w-lagr-conn-coords-curv})$ in coordinates proves that the
  curvature of $\nabla$ vanishes identically.
\end{proof}

\begin{exmp}
  \label{ex:s2w-rays}
  Let $S^{n-1}\subseteq\Rb{n}$ be a unit sphere centered at the origin and let
  $\omega$ be a standard symplectic form on its cotangent bundle. We will
  interpret this $2$-form as a restriction of the ambient symplectic form
  $\omega = d(\sum_{i=1}^n q_i\,dp_i)=\sum_{i=1}^n\,dq_i\wedge dp_i$ on
  $T^*\Rb{n}$ to the submanifold
  \begin{equation}
    T^*S^{n-1} = \smset{(p,q)\in T^*\Rb{n} : \smsum{i} q_ip_i = 0\land\smsum{i}
    p_i^2 = 1}
  \end{equation}
  induced by the embedding of $S^{n-1}$ into $\Rb{n}$. Note that we adopt the
  notational convention of geometric optics as in \cite[Chapter 3,
  1.5]{arnoldsympl}, where the usual meanings of variables $p$, $q$ is
  interchanged: the momentum vector $p$ is an element of $S^{n-1}$ which marks
  the oriented direction of a~ray in $\Rb{n}$, while $q\in T^*_pS^{n-1}$ is the
  covector dual to the perpendicular displacement $v\in T_pS^{n-1}$ of the ray
  with respect to the standard metric $\langle v,w\rangle = \smsum{i}v_iw_i$ on
  $\Rb{n}$. In this manner we construct a bijection between the pairs $(p,q)\in
  T^*S^{n-1}$ and oriented affine lines in $\Rb{n}$. (While it is equally
  possible to use the more direct identification of the space of oriented
  affine lines with the tangent bundle $TS^{n-1}$, the natural symplectic
  structure of the cotangent bundle makes $T^*S^{n-1}$ a more natural
  setting for our discussion.)

  All rays $(p,q)$ that cross a given point $x\in\Rb{n}$ comprise a Lagrangian
  submanifold
  \begin{equation}
    L_x=\smset{(p,q)\in T^*S^{n-1} : q = x - \langle p, x\rangle p}
  \end{equation}
  of $(T^*S^{n-1},\omega)$, since on $L_x$ we have
  \begin{equation}
    \begin{aligned}
      \omega = \smsum{i}dq_i\wedge dp_i
      &= - d\langle p,x \rangle\wedge (\smsum{i} p_i\,dp_i)
        - \langle p,x \rangle(\smsum{i}dp_i\wedge dp_i) \\
      &= - d\langle p,x \rangle\wedge d(\smfrac{1}{2}\abs{p}^2) = 0.
    \end{aligned}
  \end{equation}
  In particular, any hypersurface $H\subseteq\Rb{n}$ which intersects a fixed
  ray $(p_0,q_0)\in T^*S^{n-1}$ transversely at a single point $x\in\Rb{n}$
  determines a Lagrangian foliation $\gol$ of an open neighbourhood $U$ of
  $(p_0,q_0)\in T^*S^{n-1}$. Given two such disjoint hypersurfaces $H,K$,
  we obtain an interesting example of a bi-Lagrangian structure
  $(U,\omega,\fol,\gol)$ provided by Tabachnikov in \cite[274]{2-webs}. Foliations
  $\fol,\gol$ indeed form a $2$-web, since any two points $x\in H$ and $y\in K$
  determine uniquely a~single ray $(p,q)\in T^*S^{n-1}$ oriented from $y$ to
  $x$, where
  \begin{equation}
    p = \frac{x-y}{\abs{x-y}}\qquad\text{and}\qquad
    q = x - \langle x-y, x\rangle\,\frac{x-y}{\abs{x-y}^2}.
  \end{equation}
  We can prove by a direct calculation that a vector
  $v=\smsum{i}v_{x,i}\basis{x_i} +\smsum{j}v_{y,j}\basis{y_j}$ satisfies
  $dp(v)=dq(v)=0$ if and only if it spanned by the tangent vector fields
  $(\smsum{i}(x_i-y_i)\basis{x_i}, \smsum{j}(x_j-y_j)\basis{y_j})$. Hence, when
  a ray $(p_0,q_0)$ is transverse to both hypersurfaces, any nonsingular smooth
  parametrization $(t,s)\mapsto (x(s),y(t))$ by parameters $s,t\in\Rb{n-1}$ of
  points $x\in H$, $y\in K$ lying on the hypersurfaces yields a local
  coordinate system satisfying $T\fol=\bigcap_{i=1}^{n-1}\ker dt_i$ and
  $T\gol=\bigcap_{j=1}^{n-1}\ker ds_j$ for the corresponding Lagrangian
  foliations $\fol,\gol$. With its help we can obtain an expression for the
  curvature of the bi-Lagrangian connection $\nabla$ associated to this
  structure.

  The above parametrization allows us to identify an open neighbourhood of the
  ray $(p,q)\in T^*S^{n-1}$ with $H\times K$. The symplectic form $\omega$ on
  $H\times K\subseteq T^*S^{n-1}$ in the above coordinates becomes
  \begin{equation}
    \label{eq:s2w-lagr-rays-omega-intermediate}
    \begin{aligned}
      \omega &= \smsum{i}dq_i\wedge dp_i \\
        &= \smsum{i} dx_i\wedge dp_i - d(\langle p,x\rangle)\wedge (\smsum{i}
          p_idp_i) - \langle p,x\rangle (\smsum{i} dp_i\wedge dp_i) \\
        &= \smsum{i} dx_i\wedge dp_i = \smsum{i}dx_i\wedge
          d\big(\smfrac{x_i-y_i}{\abs{x-y}}\big) \\
        &= -(\smsum{i} \smfrac{1}{\abs{x-y}} dx_i\wedge dy_i)
          -\smsum{i}dx_i\wedge\big((x_i-y_i)(\smsum{j}\smfrac{(x_j-y_j)}{\abs{x-y}^{3}}(dx_j-dy_j)\big),
    \end{aligned}
  \end{equation}
  which yields
  \begin{equation}
    \label{eq:s2w-lagr-rays-omega}
     \omega = \smfrac{1}{\abs{x-y}^{3}}\big(\smsum{i}(x_i-y_i)dx_i\wedge
      \smsum{j}(x_j-y_j)dy_j\big) -
      \smfrac{1}{\abs{x-y}}\smsum{i}dx_i\wedge dy_i,
  \end{equation}
  where $x=x(s)$, $y=y(t)$ are smooth functions in parameters
  $s=(s_1,s_2,\ldots,s_{n-1})\in\Rb{n-1}$ and
  $t=(t_1,t_2,\ldots,t_{n-1})\in\Rb{n-1}$. Using these coordinates one obtains
  the matrix $A$ of symplectic form $\omega$ with entries
  $A_{ij}=\omega(\basis{s_i},\basis{t_j})$ for $i,j=1,2,\ldots,n-1$, which
  allows us to compute the matrix of curvature $2$-forms of $\nabla$ by means
  of Proposition $\ref{thm:s2w-lagr-conn-coords}$. Denote the Jacobi matrices
  of $x(s),y(t)$ by $\pd{x}{s},\pd{y}{t}\in M_{n\times(n-1)}(\Rb{1})$. By
  treating $(x-y)$ as column vectors, equality $(\ref{eq:s2w-lagr-rays-omega})$
  reduces to
  \begin{equation}
    \label{eq:s2w-lagr-ex-A}
    A = (\smpd{x}{s})^T\Big(\smfrac{1}{\abs{x-y}}\big(
            \smfrac{(x-y)}{\abs{x-y}}\smfrac{(x-y)}{\abs{x-y}}^T - I
                \big)\Big)\smpd{y}{t}.
  \end{equation}
  To find the curvature it is necessary to invert $A$. This task becomes quite
  difficult in full generality. Below we consider only the case $n=2$ for
  clarity.

  For $n=2$, the mappings $x(s),y(t)$ are smooth curves with derivatives
  $x'(s), y'(t)$ respectively, while the symplectic form $\omega$ becomes
  $f(s,t)\,ds\wedge dt$ in coordinates for some smooth function $f\in
  C^\infty(\Rb{2})$. From equality $(\ref{eq:s2w-lagr-ex-A})$ we deduce
  \begin{equation}
    \begin{aligned}
      \omega &= \smfrac{1}{\abs{x-y}^3}\big(\langle x-y,x'\rangle\langle
        x-y,y'\rangle-\langle x-y,x-y\rangle\langle x',y'\rangle\big)\ ds\wedge
        dt\\
        &= \smfrac{1}{\abs{x-y}^3}\big(
          (\basis{s}\smfrac{1}{2}\abs{x-y}^2)
          (-\basis{t}\smfrac{1}{2}\abs{x-y}^2)
          -
          \abs{x-y}^2
          (-\smfrac{\partial^2}{\partial s\partial t}\smfrac{1}{2}\abs{x-y}^2)
        \big)\ ds\wedge dt \\
        &= \big(\smfrac{1}{4\abs{x-y}^3}(\smpdd{}{s}{t}\log(\abs{x-y}^2)) +
        \smfrac{1}{4\abs{x-y}}(\smpdd{}{s}{t}\abs{x-y}^2)\big)\ ds\wedge dt.
    \end{aligned}
  \end{equation}
  We arrive at the only non-zero curvature coefficient $\xi$ of $\nabla$ by
  taking the mixed second logarithmic partial derivative of the above
  coefficient with respect to $s,t$. The result for $n=2$, obtained with the
  help of a~computer algebra system (\texttt{Wolfram Mathematica 13}
  \cite{math}), is
  \begin{equation}
    \label{eq:s2w-rays-curv}
    \begin{aligned}
      \xi &= -\frac{6\langle x-y,x'\rangle\langle x-y,y'\rangle}{\abs{x-y}^4}
          +\frac{3\langle x',y'\rangle}{\abs{x-y}^2}
          +\frac{\det(x',y')\det(x'',x-y)}{\det(x',x-y)^2} \\
        & -\frac{\det(x'',y')}{\det(x',x-y)}
          -\frac{\det(y',x')\det(y'',x-y)}{\det(y',x-y)^2}
          +\frac{\det(y'',x')}{\det(y',x-y)},
    \end{aligned}
  \end{equation}
  where $\det(v,w)$ for $v,w\in\Rb{2}$ denotes the determinant of a square
  matrix formed by concatenating the two column vectors $v,w$.

  For further reference, we also compute the volume form $\omega^{n-1}$ in
  coordinates $(s_1,\ldots,s_{n-1},$ $t_1,\ldots,t_{n-1})$ in the general case.
  Given an arbitrary matrix $A=[a_{ij}]_{i,j=1,\ldots,n}$ and a $2$-form
  $\hat\omega = \sum_{i,j=1}^n a_{ij}\,dx_i\wedge dy_j$ on $\Rb{2n}$, one can
  readily verify that
  \begin{equation}
    \hat\omega^{n-1} = (-1)^{(n-1)(n-2)/2}(n-1)! \sum_{i,j=1}^n \det
    A_{i,j}\,d\hat{x}_i\wedge d\hat{y}_j,
  \end{equation}
  for $d\hat{x}_i=dx_1\wedge\cdots\wedge dx_{i-1}\wedge
  dx_{i+1}\wedge\cdots\wedge dx_{n}$ and $d\hat{y}_j=dy_1\wedge\cdots\wedge
  dy_{j-1}\wedge dy_{j+1}\wedge\cdots\wedge dy_{n}$, where $A_{i,j}$ denotes
  the matrix $A$ with $i^\text{th}$ row and $j^\text{th}$ column discarded.
  Assume now that the matrix $A$ has the form $c(uv^T - I)$ for some column vectors
  $u,v\in\Rb{n}$ and $c\in\Rb{1}$, as in our case with $c=\frac1{\abs{x-y}}$
  and $u=v=\frac{x-y}{\abs{x-y}}$. The expressions $\det A_{i,i}$ for
  $i=1,2,\ldots,n$ is easily computed using the characteristic polynomial
  of a rank $1$ matrix $\tilde{u}_i\tilde{v}_i^T$, where $\tilde{u}_i$ and
  $\tilde{v}_i$ are the column vectors $u,v$ with $i^\text{th}$ coordinates
  removed. Since the kernel of $\tilde{u}_i\tilde{v}_i^T$ has dimension $n-2$,
  it has eigenvalue $0$ with geometric multiplicity $n-2$, hence its
  characteristic polynomial $\chi(\lambda)=\det(\lambda
  I-\tilde{u}_i\tilde{v}_i^T)$ is divisible by $\lambda^{n-2}$. Now use the
  fact that the trace of a matrix is the sum of its eigenvalues to arrive at
  \begin{equation}
    \chi(\lambda) = \lambda^{n-2}(\lambda-\operatorname{tr}(\tilde{u}_i\tilde{v}_i^T))
      = \lambda^{n-2}(\lambda-\smsum{k\neq i}u_kv_k).
  \end{equation}
  Since $\det A_{i,i}$ is exactly $(-c)^{n-1}\chi(1)$, we obtain that
  \begin{equation}
    \det A_{i,i} = (-1)^{n-1}c^{n-1}(u_iv_i - (\langle u, v\rangle - 1)).
  \end{equation}
  To compute $\det A_{i,j}$ for $i\neq j$, assume without loss of generality
  that $i<j$. In this case the $i^\text{th}$ column of $A_{i,j}$ is exactly
  $cv_i\tilde{u}_i\in\Rb{n-1}$. By the multilinearity and skew-symmetry of
  $\det A_{i,j}$ with respect to the columns of its argument we obtain that
  $\det A_{i,j} = c^{n-1}v_i\det B_{i,j}$ for
  \begin{equation}
    B_{i,j}=\begin{bmatrix}
      -1& \cdots & 0 & u_1 & 0 &\cdots &0 \\
      \vdots & \ddots & \vdots & \vdots & \vdots & \ddots & \vdots \\
      0& \cdots & -1 & u_{i-1} & 0 &\cdots &0 \\
      0& \cdots & 0 & u_{i+1} & -1 &\cdots &0 \\
      \vdots & \ddots & \vdots & \vdots & \vdots & \ddots & \vdots \\
      0& \cdots & 0 & u_{j} & 0 &\cdots &0 \\
      \vdots & \ddots & \vdots & \vdots & \vdots & \ddots & \vdots \\
      0& \cdots & 0 & u_{n} & 0 &\cdots &-1 \\
    \end{bmatrix},
  \end{equation}
  using elementary column operations on matrices, where the $(j-1)^\text{th}$ row
  containing $u_{j}$ has only one non-zero entry. By applying the Laplace
  expansion to this row we arrive at
  \begin{equation}
    \det A_{i,j} = (-1)^{n-1+i+j}c^{n-1}u_jv_i.
  \end{equation}
  To summarize, for any $i,j=1,2,\ldots,n$ we have obtained
  \begin{equation}
    \det A_{i,j} = (-1)^{n-1}c^{n-1}\Big((-1)^{i+j}u_jv_i - \delta_{ij}(\langle u,
    v\rangle - 1)\Big),
  \end{equation}
  where $\delta_{ij}$ is equal to $1$ if $i=j$ and $0$ otherwise. In our
  case $c=\frac1{\abs{x-y}}$ and $u=v=\frac{x-y}{\abs{x-y}}$, hence the above
  expression reduces to
  \begin{equation}
    \det A_{i,j} = (-1)^{n-1+i+j}\frac{(x_i-y_i)(x_j-y_j)}{\abs{x-y}^{n+1}}.
  \end{equation}
  Inserting these coefficients into the expression for $\omega^{n-1}$ leads to
  \begin{equation}
    \begin{aligned}
      \omega^{n-1} &= (-1)^{n(n-1)/2}(n-1)!\,\sum_{i,j=1}^n (-1)^{i+j}
          \frac{(x_i-y_i)(x_j-y_j)}{\abs{x-y}^{n+1}}\ d\hat{x}_i\wedge
            d\hat{y}_j\\
        &= (-1)^{n(n-1)/2}(n-1)!\,\sum_{i,j=1}^n (-1)^{i+j}
          \det\big(\smfrac{\partial x}{\partial s}\big)_i
          \frac{(x_i-y_i)(x_j-y_j)}{\abs{x-y}^{n+1}}
          \det\big(\smfrac{\partial y}{\partial t}\big)_j \\
          &\hphantom{\hskip 2em\Big(\sum_{j=1}^n (-1)^{j+n-1}
              \det\big(\smfrac{\partial y}{\partial t}\big)_j
            (x_j-y_j)\Big)} \ ds_1\wedge\cdots\wedge ds_{n-1}
            \wedge dt_1\wedge\cdots\wedge dt_{n-1} \\
        &= (-1)^{n(n-1)/2}\frac{(n-1)!}{\abs{x-y}^{n+1}}\,
          \Big(\sum_{i=1}^n (-1)^{i+n-1}
            \det\big(\smfrac{\partial x}{\partial s}\big)_i
            (x_i-y_i)\Big) \\
          &\hskip 2em\cdot\Big(\sum_{j=1}^n (-1)^{j+n-1}
              \det\big(\smfrac{\partial y}{\partial t}\big)_j
            (x_j-y_j)\Big)
          \ ds_1\wedge\cdots\wedge ds_{n-1}
            \wedge dt_1\wedge\cdots\wedge dt_{n-1},
    \end{aligned}
  \end{equation}
  where $\big(\frac{\partial x}{\partial s}\big)_i, \big(\frac{\partial
  y}{\partial t}\big)_j$ denote the Jacobi matrices
  of $x(s), y(t)$ with $i^\text{th}, j^\text{th}$ row deleted respectively.
  Note that the two factors involving the determinants in the last two lines of
  the above equality are exactly the Laplace expansions of determinants
  $\det(\smpd{x}{s},x-y), \det(\smpd{y}{t},x-y)$ of the Jacobi matrices
  concatenated with the column vector $x(s)-y(t)$. Therefore, we can write
  \begin{equation}
    \label{eq:s2w-lagr-rays-vol}
      \omega^{n-1} = f(s,t)\,ds_1\wedge\cdots\wedge ds_{n-1}\wedge
      dt_1\wedge\cdots\wedge dt_{n-1}
  \end{equation}
  with
  \begin{equation}
      f(s,t) = (-1)^{n(n-1)/2}(n-1)! \
      \frac{\det(\smpd{x}{s},x-y)\det(\smpd{y}{t},x-y)}{\abs{x-y}^{n+1}}
  \end{equation}
  The above expression is valid for any parametrizations
  $(x_1(s),\ldots,x_n(s))$ and $(y_1(t),\ldots,y_n(t))$ of $H,K$. It is
  non-zero if $x-y$ is transverse to $H$ and $K$, which proves that
  $(U,\omega,\fol,\gol)$ is a (regular) bi-Lagrangian structure if and only if
  $H\cap K=\varnothing$ and $(p_0,q_0)$ intersects $H$ and $K$ transversely. By
  taking the natural logarithm of $f(s,t)$ and differentiating it with respect
  to $s_i$ and $t_j$ we obtain the Ricci tensor of the canonical connection
  $\nabla$.
\end{exmp}

\subsection{Bi-Lagrangian submanifolds}

Throughout this section, the symbol $\web_\omega$ will denote a fixed
bi-Lagrangian structure $(M,\omega,\fol,\gol)$ with bi-Lagrangian connection
$\nabla$. Our current goal is to describe smooth submanifolds $S\subseteq M$
which admit a bi-Lagrangian structure canonically induced from $M$. To simplify
notation, we will use the restriction symbol $E_{|S}$ to denote the pullback
bundle $\iota^*E$ of any given vector subbundle $E\hookrightarrow TM
\twoheadrightarrow M$ of $TM$ along the corresponding inclusion
$\iota:S\hookrightarrow M$.

\begin{defn}
  A submanifold $S\subseteq M$ is called \emph{a bi-Lagrangian submanifold of
  $\web_\omega$} if the restrictions $T\pol$, $T\qol$ of $T\fol$, $T\gol$ to
  $TS$ integrate to nonsingular foliations $\pol$, $\qol$ of $S$ and the
  quadruple $\web_{\omega|S}=(S,\omega_{|S},\pol,\qol)$ forms a Lagrangian
  $2$-web.
\end{defn}

We list a couple of elementary consequences of this definition. Since
$\omega_{|S}$ is nondegenerate on $S$, the dimension of $S$ has to be an even
number. The tangent bundle $TS$ decomposes as a direct sum
$T\pol\oplus T\qol$, where $\dim T\pol = \dim T\qol = \frac{1}{2}\dim S$ due to
$\pol$,$\qol$ being Lagrangian foliations.
The assumption that $S$ is a symplectic submanifold provides us with a direct
sum decomposition $TM_{|S}=TS\oplus TS^\omega$ into $TS$ and its
skew-orthogonal complement $TS^\omega$ consisting of vectors $v\in TM$ such
that $\omega(v,\cdot)_{|TS}=0$. It defines the canonical skew-orthogonal
projection $p_S:TM_{|S}\twoheadrightarrow TS$, where, given $v\in TM_{|S}$, the
vector $p_Sv$ can be characterized as the unique vector from $TS$ that
satisfies
\begin{equation}
  \label{eq:s2w-lagr-skewproj}
  \omega(v,w) = \omega(p_S\,v,w) \qquad\text{ for all $w\in TS$}.
\end{equation}
Note that vectors from $T\fol_{|S}$ ($T\gol_{|S}$) project down to
$T\pol$ ($T\qol$), since in this case $\omega(v,\cdot)$ vanishes on $T\pol$
($T\qol$). The images of these projections are connected by canonical
isomorphisms $T\pol\simeq T\qol^*$, $T\qol\simeq T\pol^*$ defined as in
$(\ref{eq:s2w-lagr-isos})$ by taking $v\mapsto \iota_v\omega$. If we denote
them by $\alpha_S$, $\beta_S$ respectively, then it is apparent that
$(\alpha v)_{|S}=\alpha_Sv$ and $(\beta w)_{|S}=\beta_Sw$ for $v\in T\pol$ and
$w\in T\qol$.

Let $\nabla^S$ be the bi-Lagrangian connection of the Lagrangian subweb $S$ of
$\web_\omega$. The relationship between $\nabla^S$ and $\nabla$ can be
clarified using pseudo-Riemannian techniques relying on a known
correspondence between bi-Lagrangian geometry and para-Kähler geometry
(see e.g. \cite{etayosantamaria,bilagrangian}).

Each bi-Lagrangian structure $(M,\omega,\fol,\gol)$ of dimension $2n$ carries a
canonical metric $g$ of signature $(n,n)$ obtained in the following way. The
underlying pair of foliations $\fol$, $\gol$ gives rise to an integrable
almost-product structure $J$ by taking $Jv_\fol=v_\fol$ and $Jv_\gol=-v_\gol$
for $v_\fol\in T\fol, v_\gol\in T\gol$. This almost-complex structure $J$ has
the additional property that its eigenvalues $\pm1$ occur with the same
multiplicity $n$; we call such integrable almost product structures
\emph{para-complex structures}. To define $g$, for each $v,w\in TM$ put
\begin{equation}
  \label{eq:s2w-lagr-metr}
  g(v,w) = \omega(Jv,w).
\end{equation}
It can be proved by a straightforward calculation that $\nabla J=0$ as a
consequence of property $(\ref{def:s2w-lagr-conn:fol})$ of Definition
\ref{def:s2w-lagr-conn}. This, together with $\nabla\omega=0$, yields $\nabla
g=0$. Since $\nabla$ is torsionless, the connection $\nabla$ coincides with the
Levi-Civita connection of $(M,g)$. All of the above properties allow us to
deduce that the triple $(M,g,J)$ forms a \emph{para-Kähler manifold}
\cite{paracomplex}: a structure consisting of a smooth manifold $M$ equipped
with a \emph{para-complex structure} $J$ and a \emph{neutral metric} $g$ with
Levi-Civita connection $\nabla$ satisfying $\nabla J=0$ and $g(Jv,Jw)=-g(v,w)$
for each $v,w\in TM$.

It is of note that the tangent projection $\maps{p_S}{TM_{|S}}{TS}$ onto a
bi-Lagrangian submanifold $S$ given by $(\ref{eq:s2w-lagr-skewproj})$ is equal to the
orthogonal projection of $TM$ onto $TS$ with respect to the induced metric $g$.
Indeed, for each $v\in TM_{|S}$ and $w\in TS$ the identity
\begin{equation}
  g(v,w) = -\omega(v,Jw) = -\omega(p_S\,v,Jw) = g(p_S\,v,w)
\end{equation}
holds by the symmetry of $g$. Since $\nabla$ is Levi-Civita and $p_S$ is
orthogonal, the classical theory translated into the bi-Lagrangian language
yields the following formula for the canonical connection $\nabla^S$ on $S$.

\begin{prop}
  \label{thm:s2w-lagr-conn-proj}
  Let $\nabla$ be the canonical connection of a bi-Lagrangian manifold
  $(M,\omega,\fol,\gol)$, and let $\nabla^S$ be the canonical connection of one
  of its bi-Lagrangian submanifolds $S$. Then
  \begin{equation}
    \pushQED{\qed}
    \nabla^S = p_S\circ\nabla.\qedhere
    \popQED
  \end{equation}
\end{prop}

\medskip
We will now state some results regarding bi-Lagrangian submanifolds drawn from
the pseudo-Riemannian world by means of the above characterization of
$\nabla^S$. The most important one for our purposes is the bi-Lagrangian
analogue of the Gauss equation relating the curvature of a surface to the
curvature of its ambient space \cite{docarmo, foundg2, renteln}.

Its formulation in the bi-Lagrangian language relies on the notion of a
\emph{symplectic curvature tensor} \cite{vaisman}. This covariant $4$-tensor
$Rs$ is defined in a familiar way using the Riemann curvature endomorphism
$R(u,v)w = \nabla_u\nabla_v w - \nabla_v\nabla_u w - \nabla_{[u,v]}w$, namely
\begin{equation}
  \label{eq:s2w-lagr-rs}
  Rs(X,Y,Z,W) = \omega(R(Z,W)Y,X).
\end{equation}
for each $X,Y,Z,W\in\mathfrak{X}(M)$. It exhibits several symmetries similar to
those underlying the classical Riemann curvature tensor \cite{vaisman} in
addition to some other symmetries involving the projections $X=X_\fol+X_\gol$,
where $X_\fol\in\Gamma(T\fol)$ and $X_\gol\in\Gamma(T\gol)$ for a fixed
$X\in\mathfrak{X}(M)$ \cite{hess},
\begin{enumerate}[label=$(\alph{enumi})$, ref=\alph{enumi}]
  \item\label{thm:s2w-lagr-symmetries:2form}
    $Rs(X,Y,Z,W) = -Rs(X,Y,W,Z)$,
    \hfill\emph{(antisymmetry of curvature $2$-forms)}
  \item\label{thm:s2w-lagr-symmetries:bianchi}
    $Rs(X,Y,Z,W) + Rs(X,Z,W,Y) + Rs(X,W,Y,Z) = 0$,
    \hfill\emph{(algebraic Bianchi identity)}
  \item\label{thm:s2w-lagr-symmetries:omega}
    $Rs(X,Y,Z,W) = Rs(Y,X,Z,W)$,
    \hfill\emph{($R(Z,W)$-invariance of $\omega$)}
  \item\label{thm:s2w-lagr-symmetries:flat}
    $Rs(X_\fol,Y_\fol,Z,W) = Rs(X_\gol,Y_\gol,Z,W) = 0$,
    \hfill\emph{(flatness along $T\fol$, $T\gol$)}
  \item\label{thm:s2w-lagr-symmetries:fols}
    $Rs(X,Y,Z_\fol,W_\fol) = Rs(X,Y,Z_\gol,W_\gol) = 0$.
    \hfill\emph{($\nabla$ preserves $T\fol$, $T\gol$)}
\end{enumerate}
The neutral metric $g$ arising out of the bi-Lagrangian structure $\web_\omega$
via $(\ref{eq:s2w-lagr-metr})$ gives rise to the standard Riemann curvature
tensor $Rm$, which is related to $Rs$ by the equality
\begin{equation}
  \label{eq:s2w-lagr-rsrm}
  Rs(X,Y,Z,W) = Rm(-JX,Y,Z,W),
\end{equation}
where $J$ is the almost-product structure coming from $\fol, \gol$. This
relationship, in conjunction with the classical Gauss equation, makes it
straightforward to prove the \emph{bi-Lagrangian Gauss equation} linking the
symplectic curvature tensor $Rs^S$ of a bi-Lagrangian submanifold $S$ with its
ambient counterpart $Rs$. If we denote the \emph{second fundamental form} of
$\nabla$ by
\begin{equation}
  \label{eq:s2w-lagr-II}
  \II(v,w) = \nabla_vw - \nabla^S_vw,
\end{equation}
the
equation says that
\begin{equation}
  \label{eq:s2w-lagr-gauss}
  \begin{aligned}
    &Rs(X,Y,Z,W) = Rs^S(X,Y,Z,W) \\
      &\hskip 4em + \omega(\II(X,Z),\II(Y,W)) - \omega(\II(X,W),\II(Y,Z))
  \end{aligned}
\end{equation}
for each $X,Y,Z,W\in\mathfrak{X}(S)$.

\subsection{Geometric flatness conditions}

The correspondence between bi-Lagrangian and para-Kähler geometry given by the
metric $g$ in $(\ref{eq:s2w-lagr-metr})$ suggests that we can extract all the
information about the curvature of the bi-Lagrangian manifold
$\web_\omega=(M,\omega,\fol,\gol)$ from the curvature of suitable immersed
$2$-dimensional subwebs. In the metric case, the relevant notion is that of
\emph{sectional curvature}. Here, we rely on a certain class of bi-Lagrangian
submanifolds locally spanned by a pair of geodesics with respect to the
bi-Lagrangian connection $\nabla$ to recover the curvature of $\web_\omega$. We
now give more details on these surfaces.

Locally, say, in a neighbourhood of a point
$p\in M$, we can express $M$ as a product of two leaves $F\in\fol$ and
$G\in\gol$ intersecting at $p$. Since an immersion of a subweb preserves the
corresponding foliations, the germ of immersion $\iota_S$ of a $2$-dimensional
bi-Lagrangian submanifold $S$ into $\web_\omega$ must be a product of curves
$\gamma_F\times\gamma_G$, where $\maps{\gamma_F}{(\Rb{1},0)}{F\in\fol}$ and
$\maps{\gamma_G}{(\Rb{1},0)}{G\in\gol}$, with
$\omega(\dot\gamma_F,\dot\gamma_G)\neq 0$ and $\gamma_F(0)=\gamma_G(0)=p$.

\begin{defn}
  \label{def:s2w-lagr-gen}
  Let $p\in M$, and let $F\in\fol$, $G\in\gol$ be the leaves of the
  bi-Lagrangian structure $\web_\omega=(M,\omega,\fol,\gol)$ crossing $p$.
  Given two smooth functions $H,K\in C^\infty(M)$, a $2$-dimensional
  bi-Lagrangian submanifold $S\subseteq M$ of $\web_\omega$ is said to be
  \emph{generated by Hamiltonians $H$, $K$ at $p$} if
  \begin{enumerate}[label=$(\alph{enumi})$, ref=\alph{enumi}]
    \item $dH_{|T\fol}=0$, $dK_{|T\gol}=0$,
    \item $\omega_p(X_H,X_K)\neq 0$,
    \item the leaves of $S$ crossing $p$ are the images of the
  integral curves $\maps{\gamma_F}{(\Rb{1},0)}{F}$ and
  $\maps{\gamma_G}{(\Rb{1},0)}{G}$ of the Hamiltonian vector fields
  $X_H,X_K\in\mathfrak{X}(M)$ corresponding to $H,K$.
  \end{enumerate}
  In this case we say that the bi-Lagrangian surface $S$ generated by $H,K$ at
  $p$ is \emph{spanned} by $\gamma_F$ and $\gamma_G$.
\end{defn}

The two curves $\gamma_F,\gamma_G$ are indeed $\nabla$-geodesics. This fact,
which follows from equality $(\ref{eq:s2w-lagr-leafwise-const})$ as
demonstrated in the proof of the next lemma, leads to the equality
between the only non-zero coefficient of the symplectic curvature tensor
$Rs^S_p$ of $\nabla^S$ at $p\in S$ and the corresponding coefficient of the
ambient curvature tensor $Rs_p$ defined in $(\ref{eq:s2w-lagr-rs})$.

\begin{lem}
  \label{thm:s2w-lagr-slicecore}
  If $\gamma_F$ is an integral curve of a Hamiltonian flow corresponding to
  a Hamiltonian $H$ such that $dH_{|T\fol}=0$, then
  \begin{equation}
    \label{eq:s2w-lagr-subweb-rs}
    Rs_p(\dot{\gamma}_G, \dot{\gamma}_F, \dot{\gamma}_F,
      \dot{\gamma}_G)
    = Rs^S_p(\dot{\gamma}_G, \dot{\gamma}_F, \dot{\gamma}_F,
      \dot{\gamma}_G).
  \end{equation}
\end{lem}
\begin{proof}
  Note that $\dot\gamma_F\in T\fol$, since $T\fol$ is a Lagrangian subspace
  of $(TM,\omega_p)$ and $\omega(\dot\gamma_F,\cdot) = dH$. Now, the equality
  $dH_{|T\fol}=0$ leads via $(\ref{eq:s2w-lagr-leafwise-const})$ to
  $\nabla_{\dot\gamma_F} dH = 0$. Use Proposition \ref{thm:s2w-lagr-formula} to
  obtain
    $\nabla_{\dot\gamma_F}\dot\gamma_F =
    \nabla_{\dot\gamma_F}\alpha^{-1}dH =
    \alpha^{-1}\nabla_{\dot\gamma_F}dH = 0$.
  This implies $\II(\dot\gamma_F,\dot\gamma_F) = 0$. Since $\II(v,w) = 0$ for
  every $v\in T_p\pol,w\in T_p\qol$, an application of
  $(\ref{eq:s2w-lagr-gauss})$ proves the claim.
\end{proof}

The proof of the Proposition below provides a coordinate-free construction of
bi-Lagrangian surfaces generated by Hamiltonians.

\begin{lem}
  \label{thm:s2w-lagr-surf}
  Let $\web_\omega=(M,\omega,\fol,\gol)$ be a bi-Lagrangian manifold. Given any
  point $p\in M$ and any pair of tangent vectors $v\in T_p\fol$, $w\in T_p\gol$
  with $\omega(v,w)\neq 0$ there exists a bi-Lagrangian surface
  $S\subseteq M$ generated by Hamiltonians $H$, $K$ at $p$ such that the
  integral curves $\gamma_H,\gamma_K$ of the corresponding Hamiltonian vector
  fields $X_H,X_K\in\mathfrak{X}(M)$ crossing $p$ are exactly the geodesics of
  $\nabla$ satisfying $\dot{\gamma}_H(0)=v$ and $\dot{\gamma}_K(0)=w$.
\end{lem}
\begin{proof}
  Let $F$, $G$ be the leaves of $\fol$, $\gol$ crossing $p$ inside a
  sufficiently small open neighbourhood $U$ of $p$ and let
  $\eta=\omega(v,\cdot)$, $\xi=\omega(\cdot,w)\in T_p^*M$. Pick any smooth function
  $\tilde{H}\in C^\infty(G)$ on $G$ such that $\eta_{|TG}=d\tilde{H}_{|p}$ and
  extend it to a function $H\in C^\infty(M)$ which is constant on the leaves of
  $\fol$ inside the neighbourhood $U$. This property guarantees that the
  function $H$ satisfies $dH_{|T\fol} = 0$ and $\eta = dH_{|p}$, since $v\in
  T\fol=T\fol^\omega$. One similarly
  constructs the other function $K\in C^\infty(M)$ so that $dK_{|T\gol} = 0$
  and $\xi = dK_{p}$. The corresponding Hamiltonian vector fields satisfy
  \begin{equation}
    \begin{aligned}
      \omega(X_{H|p},\cdot) = dH_{|p} = \eta = \omega(v,\cdot),&\quad\text{
        hence }\quad X_{H|p} = v,\\
      \omega(\cdot,X_{K|p}) = dK_{|p} = \xi = \omega(\cdot,w),&\quad\text{
        hence }\quad X_{K|p} = w,
    \end{aligned}
  \end{equation}
  and, for each $Y\in\Gamma(T\fol)$ and $Z\in\Gamma(T\gol)$,
  \begin{equation}
    \begin{aligned}
      \omega(X_H,Y) = dH(Y) = 0,&\quad\text{ hence }\quad
      X_{H|q}\in T_q\fol^\omega=T_q\fol \text{ for each } q\in U,\\
      \omega(Z,X_K) = dK(Z) = 0,&\quad\text{ hence }\quad
      X_{K|q}\in T_q\gol^\omega=T_q\gol \text{ for each } q\in U.\\
    \end{aligned}
  \end{equation}

  Restrict the vector fields $X_H,X_K$ to the leaves $F,G$ of $\fol,\gol$
  crossing $p$ respectively. Recall that, by Lemma \ref{thm:s2w-lagr-frames},
  the connection $\nabla$ is flat on leaves of $\fol$ and $\gol$. Thus, the
  vector fields $X_{H|F},X_{K|G}$ extend to smooth vector fields $Y,Z$ defined
  in an open neighbourhood of $p$ which are $\nabla$-parallel along the leaves
  of $\gol,\fol$, as smoothly parametrized families of $\nabla$-parallel
  extensions of individual tangent vectors $X_{H|q},X_{K|q'}$ along the leaves
  of $\gol,\fol$ crossing $q\in F, q'\in G$ respectively.

  Since $\nabla_V\Gamma(T\fol)\subseteq\Gamma(T\fol)$ and
  $\nabla_V\Gamma(T\gol)\subseteq\Gamma(T\gol)$ for each $V\in\mathfrak{X}(M)$
  by property $(\ref{def:s2w-lagr-conn:fol})$ of Definition
  $\ref{def:s2w-lagr-conn}$, we have $Y\in\Gamma(T\fol)$ and
  $Z\in\Gamma(T\gol)$. This gives
  \begin{equation}
    [Y,Z] = \nabla_YZ-\nabla_ZY = 0-0 = 0,
  \end{equation}
  proving that the tangent distribution $\mathcal{D}=\langle
  Y,Z\rangle\subseteq TM$ is involutive.
  An application of Frobenius integrability theorem to $\mathcal{D}$ yields a
  foliation $\hol$ of $U$ by surfaces. Let $S\in\hol$ be the leaf of $\hol$
  crossing $p$. Since $Y\in\Gamma(T\fol\cap TS)$ and $Y_{|F\cap S} = X_{H}$, we
  obtain that the curve $F\cap S$ is an integral curve of the vector field
  $X_H$, and analogously $G\cap S$ is the integral curve of $X_K$. Since
  $\omega(X_H,X_K)=\omega(v,w)\neq 0$ and $dH_{|T\fol}=0$, $dK_{|T\gol}=0$, the
  surface $S$ is a bi-Lagrangian surface generated by Hamiltonians $H,K$ at
  $p$.
\end{proof}

Note that, given a bi-Lagrangian surface $S$ generated by Hamiltonians, the
induced symplectic form $\omega_{|S}$ is a volume form. Moreover, the
connection $\nabla^S$ preserves the volume form $\omega_{|S}$, parallelizes the
induced foliations $\pol$ and $\qol$, and is torsionless. By uniqueness in
Proposition \ref{thm:dfw-conn}, $\nabla^S$ is the natural connection
associated to the divergence-free $2$-web $\web_{\omega|S}$. This shift in
focus from symplectic to unimodular point of view opens a way to interpret the
curvature of $\nabla^S$ in affine-geometric terms using a wide variety of
geometric invariants associated with divergence-free webs.
The curvature data acquired in this way is reflected in the curvature of the
$\web_\omega$-connection $\nabla$ itself, as evidenced, for instance, by Lemma
\ref{thm:s2w-lagr-slicecore} above, highlighting the possibility to reduce the
study of $\web_\omega$ to the investigation of divergence-free web-geometric
invariants of certain surfaces in $M$. In particular, the answer to the
question of triviality of $\web_\omega$ is within the reach of these tools, as
demonstrated by the main theorem of this part of our work.

\begin{thm}
  \label{thm:s2w-lagr-equiv}
  Let $\web_\omega=(M, \omega, \fol, \gol)$ be a bi-Lagrangian manifold,
  and let $\nabla$ be its associated $\web_\omega$-connection. The
  following conditions are equivalent:
  \begin{enumerate}[label=$(\alph{enumi})$, ref=\alph{enumi}]
    \item\label{thm:s2w-lagr-main:charts}
      $M$ can be covered with coordinate charts $(x_i,y_j)_{i,j=1}^n$
      in which $T\fol = \bigcap_{i=1}^n \ker dy_i$, $T\gol = \bigcap_{i=1}^n
      \ker dx_i$ and $\omega = \sum_{i=1}^n dx_i\wedge dy_i$.
    \item\label{thm:s2w-lagr-main:flat}
      $\nabla$ is flat.
    \item\label{thm:s2w-lagr-main:subwebsflat}
      For every $2$-dimensional bi-Lagrangian submanifold $S$ generated by
      Hamiltonians, the associated $\web_{\omega|S}$-connection $\nabla^S$ is
      flat.
    \item\label{thm:s2w-lagr-main:subwebsother}
      For each point $p\in M$, every $2$-dimensional bi-Lagrangian submanifold $S$
      generated by Hamiltonians at $p$ satisfies one of the geometric
      triviality conditions of Theorem \ref{thm:dfw-geom} at $p\in S$.
  \end{enumerate}
\end{thm}
\begin{proof}
  The equivalence between $(\ref{thm:s2w-lagr-main:charts})$ and
  $(\ref{thm:s2w-lagr-main:flat})$ is known \cite{vaisman} and
  can be established using the correspondence between flatness
  (torsionlessness) of $\nabla$ and existence (commutativity) of local
  $\nabla$-parallel frames \cite[Chapter 9]{leesmooth}. Alternatively, one can
  use the coordinate formula $(\ref{eq:s2w-lagr-conn-coords-curv})$ to deduce
  $(\ref{thm:s2w-lagr-main:flat})$ from $(\ref{thm:s2w-lagr-main:charts})$ and
  obtain the converse by means of the following argument.

  Pick a point $p\in M$ and a coordinate system
  $(x_1,\ldots,x_n,y_1,\ldots,y_n)$ centered at $p\in M$. Assume that $\nabla$
  is flat. In this case, by Theorem \ref{thm:s2w-lagr-fg} there exist two
  matrix-valued function-germs $\maps{f,g}{(\Rb{n},0)}{M_{n\times n}(\Rb{n})}$
  satisfying $A(x,y) = f(x)\cdot g(y)$. Since the ambient symplectic $2$-form
  \begin{equation}
    \omega
      = \smsum{i,j}a_{ij}\,dx_i\wedge dy_j
      = \smsum{j} (\smsum{i} f_{ik}(x)\,dx_i)\wedge(\smsum{j} g_{kj}(y)\,dy_j)
  \end{equation}
  is closed, we get
  \begin{equation}
    \begin{aligned}
      0 = d\omega
        = \smsum{j} d(\smsum{i} f_{ik}(x)\,dx_i)&\wedge(\smsum{j}
        g_{kj}(y)\,dy_j) \\
        {}+ \smsum{j} (\smsum{i} f_{ik}(x)\,dx_i)&\wedge d(\smsum{j} g_{kj}(y)\,dy_j).
    \end{aligned}
  \end{equation}
  Since the two summands differ in the number of factors which annihilate
  $T\fol$, they are linearly independent, hence are both zero. By invertibility
  of $A$, and by extension $f$ and $g$, this reduces to
  \begin{equation}
    d(\smsum{i} f_{ik}(x)\,dx_i) = d(\smsum{j} g_{kj}(y)\,dy_j) = 0
      \quad\text{for }k=1,2,\ldots,n.
  \end{equation}
  By Poincare's lemma, there exist smooth function-germs $H_k,K_k$ satisfying
  $dH_k = \smsum{i} f_{ik}(x)\,dx_i$ and $dK_k = \smsum{j} g_{kj}(y)\,dy_j$
  with $H_k(0)=K_k(0)=0$ for $k=1,2,\ldots,n$. This allows us to write the
  symplectic form as
  \begin{equation}
    \omega = \smsum{k} dH_k\wedge dK_k.
  \end{equation}
  Since $\omega$ is nondegenerate, the 1-forms $dH_1, \ldots, dH_n, dK_1,
  \ldots, dK_n$ are linearly independent. Moreover, it is immediate from their
  defining formulae that these functions satisfy $H_k=H_k(x)$ and $K_k=K_k(y)$.
  Therefore, the diffeomorphism-germ
  \begin{equation}
    \varphi(x,y) = (H_1(x),\ldots,H_n(x), K_1(y),\ldots,K_n(y))
  \end{equation}
  preserves the foliations $\fol,\gol$ and carries
  $\omega$ into $\varphi^*(\omega) = \smsum{k} dx_k\wedge dy_k$; it changes the
  coordinate system into the one the existence of which was asserted in
  condition $(\ref{thm:s2w-lagr-main:charts})$.

  To deduce condition $(\ref{thm:s2w-lagr-main:subwebsflat})$ from
  $(\ref{thm:s2w-lagr-main:flat})$, fix two Hamiltonians generating the
  bi-Lagrangian surface $S$ and a coordinate system
  $(x_1,\ldots,x_n,y_1,\ldots,y_n)$ centered at $p\in M$ in which $T\fol =
  \bigcap_{i=1}^n \ker dy_i$, $T\gol = \bigcap_{i=1}^n \ker dx_i$. Let
  $\gamma_H(s)=(\tilde{x}(s),0)$ and $\gamma_K(t)=(0,\tilde{y}(t))$ be the
  integral curves of $X_H,X_K$ crossing $p$ at time $0$ and let
  $\varphi(s,t)=(\tilde{x}(s),\tilde{y}(t))$. The map $\varphi$ is a local
  parametrization of $S$. Again, Theorem \ref{thm:s2w-lagr-fg} allows us to
  write the matrix $A_{ij}=\omega(\basis{x_i},\basis{y_j})$ as
  $A(x,y)=f(x)\cdot g(y)$ for a~pair of matrix-valued function-germs
  $\maps{f,g}{(\Rb{n},0)}{M_{n\times n}(\Rb{1})}$.
  With this in hand, the vector fields $X_H, X_K$ take the form
  \begin{equation}
    \begin{aligned}
      X_H &= \smsum{i,j,k}\smpd{H}{y_i}(y)g^{-1}_{ij}(y)
        f^{-1}_{jk}(x)\,\basis{x_k}, \\
      X_K &= -\smsum{i,j,k}\smpd{K}{x_i}(x)f^{-1}_{ij}(x)
        g^{-1}_{jk}(x)\,\basis{y_k}.
    \end{aligned}
  \end{equation}
  Since $d\varphi(\basis{s})$ ($d\varphi(\basis{t})$) do not depend on the
  $y$-coordinates ($x$-coordinates), we have
  \begin{equation}
    \begin{aligned}
      d\varphi(\basis{s})_{|(s,t)} &= d\varphi(\basis{s})_{|(s,0)}
        = (X_H)_{(\tilde{x}(s),0)}
        = \smsum{i,j,k}\smpd{H}{y_i}(0)g^{-1}_{ij}(0)
          f^{-1}_{jk}(\tilde{x}(s))\,\basis{x_k}, \\
      d\varphi(\basis{t})_{|(s,t)} &= d\varphi(\basis{t})_{|(0,t)}
        = (X_K)_{(0,\tilde{y}(t))}
        = -\smsum{i,j,k}\smpd{K}{x_i}(0)f^{-1}_{ij}(0)
          g^{-1}_{jk}(\tilde{y}(t))\,\basis{y_k}.
    \end{aligned}
  \end{equation}
  Inserting these vector fields into the symplectic form
  \begin{equation}
    \omega
      = \smsum{j} (\smsum{i} f_{ik}(x)\,dx_i)\wedge(\smsum{j} g_{kj}(y)\,dy_j)
  \end{equation}
    we obtain that
  \begin{equation}
    \begin{aligned}
    \varphi^*\omega_{|(s,t)}(\basis{s},\basis{t})
      &= \omega_{|(\tilde{x}(s),\tilde{y}(t))}
        (d\varphi(\basis{s}),d\varphi(\basis{t})) \\
      &= -\smsum{i,j,k,l,m,u,v}
        (\smpd{H}{y_i}(0)g^{-1}_{ij}(0)f^{-1}_{jk}(\tilde{x}(s))
        \cdot(f_{kl}(\tilde{x}(s))g_{lm}(\tilde{y}(t)))\\
      &\hskip 10em
        \cdot(g^{-1}_{mu}(\tilde{y}(t))f^{-1}_{uv}(0)\smpd{K}{x_v}(0)) \\
      &= \omega_{|(0,0)}(X_H,X_K)
    \end{aligned}
  \end{equation}
  is constant, where the last equality follows from the assumption that the
  surface $S$ is generated by Hamiltonians $H,K$. This proves that the volume
  form $\omega_{|S}=\varphi^*\omega$ on $S$ takes the form $\omega_{|S} =
  ds\wedge dt$, which together with $\basis{s}\in T\pol$ and
  $\basis{t}\in T\qol$ implies that the connection $\nabla^S$
  associated with the divergence-free $2$-web-germ $\web_{\omega|S} =
  (S,\omega_{|S},\pol,\qol)$ is flat by Theorem \ref{thm:dfw-geom}.

  The proof of $(\ref{thm:s2w-lagr-main:subwebsother})$ given
  $(\ref{thm:s2w-lagr-main:subwebsflat})$ reduces to an application of Theorem
  \ref{thm:dfw-geom} to a subweb under consideration.

  The remaining implication from $(\ref{thm:s2w-lagr-main:subwebsother})$ to
  $(\ref{thm:s2w-lagr-main:flat})$ follows from the coincidence of symplectic
  curvature tensors of $\nabla$ and $\nabla^S$ (Lemma
  \ref{thm:s2w-lagr-slicecore}) for bi-Lagrangian surfaces $S$ generated by
  Hamiltonians at the anchor point $p\in M$. Take two tangent vectors $v\in
  T_p\fol$ and $w\in T_p\gol$ such that $\omega_p(v,w)\neq 0$ and use Lemma
  \ref{thm:s2w-lagr-surf} to find a bi-Lagrangian surface-germ $S$ at $p$ with
  tangent space $T_pS$ spanned by $v,w$. Since the
  volume-preserving holonomy of the divergence-free $2$-web $\web_{\omega|S} =
  (S,\omega_{|S},\pol,\qol)$ vanishes at $p\in S$ by
  $(\ref{thm:s2w-lagr-main:subwebsother})$, so does the curvature of its
  canonical connection $\nabla^S$ at point $p$ by combining Theorem
  \ref{thm:dfw-geom} with Lemma \ref{thm:dfw-loop-taylor}. This yields
  $Rs_p(w,v,v,w) = 0$ for every $v\in T_p\pol,w\in T_p\qol$ such that
  $\omega_p(v,w)\neq 0$.

  The proof that the vanishing of the above symplecitc analogue of sectional
  curvature implies that $Rs_p=0$ parallels the classical theory \cite[Chapter
  3]{oneillsemiriem}. Recall symmetries
  $(\ref{thm:s2w-lagr-symmetries:2form})$-$(\ref{thm:s2w-lagr-symmetries:fols})$
  of the symplectic curvature tensor $Rs$ (see p.
  \pageref{thm:s2w-lagr-symmetries:2form}).
  Assume that for each $X_\fol\in T_p\fol$ and $Y_\gol\in T_p\gol$ with
  $\omega(X_\fol,Y_\gol)\neq 0$ we have $Rs_p(X_\fol,Y_\gol,Y_\gol,X_\fol)=0$.
  Since the set of such pairs $(X_\fol,Y_\gol)$ is open and dense in
  $T_p\fol\times T_p\gol$, the equality holds also for pairs $(X_\fol, Y_\gol)$
  satisfying $\omega(X_\fol,Y_\gol)=0$ by continuity of $Rs_p$, hence we can
  drop the assumption about nonvanishing of $\omega(X_\fol,Y_\gol)$. Observe
  that, by $(\ref{thm:s2w-lagr-symmetries:flat})$ and
  $(\ref{thm:s2w-lagr-symmetries:fols})$, the bilagrangian curvature tensor
  $Rs$ vanishes whenever the first two or the last two argumetns are both in
  either $T_p\fol$ or $T_p\gol$. Using this fact, the algebraic Bianchi
  identity $(\ref{thm:s2w-lagr-symmetries:bianchi})$ of the form
  \begin{equation}
    Rs(X_\fol,Y_\gol,Z_\gol,X_\fol) + Rs(X_\fol,X_\fol,Y_\gol,Z_\gol)
    + Rs(X_\fol,Z_\gol,X_\fol,Y_\gol) = 0
  \end{equation}
  for $X_\fol\in T_p\fol$ and $Y_\gol,Z_\gol\in T_p\gol$, symmetry
  $(\ref{thm:s2w-lagr-symmetries:fols})$ applied to the second term and the
  antisymmetry $(\ref{thm:s2w-lagr-symmetries:2form})$ in the last two
  arguments of $Rs_p$, we obtain
  \begin{equation}
    \label{eq:s2w-lagr-extra-sym}
    Rs_p(X_\fol,Y_\gol,Z_\gol,X_\fol) = Rs_p(X_\fol,Z_\gol,Y_\gol,X_\fol).
  \end{equation}
  Therefore, by our assumption $Rs_p(X_\fol,Y_\gol,Y_\gol,X_\fol)=0$ we get
  \begin{equation}
    \begin{aligned}
      0 &= Rs_p(X_\fol, Y_\gol+Z_\gol, Y_\gol+Z_\gol, X_\fol) \\
        &=Rs_p(X_\fol, Y_\gol, Z_\gol, X_\fol) + Rs_p(X_\fol,Z_\gol,Y_\gol,X_\fol) \\
        &= 2Rs_p(X_\fol, Y_\gol, Z_\gol, X_\fol).
    \end{aligned}
  \end{equation}
  In the same way we obtain for arbitrary $X_\fol,W_\fol\in T_p\fol$ and
  $Y_\gol,Z_\gol\in T_p\gol$ that
  \begin{equation}
    \label{eq:s2w-lagr-sectcurv1}
    \begin{aligned}
      0 &= Rs_p(X_\fol+W_\fol, Y_\gol, Z_\gol, X_\fol+W_\fol) \\
        &= Rs_p(X_\fol, Y_\gol, Z_\gol, W_\fol) + Rs_p(W_\fol,Y_\gol,Z_\gol,X_\fol) \\
        &\overset{\smash{\mathclap{(\ref{thm:s2w-lagr-symmetries:2form})}}}{=}
           Rs_p(X_\fol, Y_\gol, Z_\gol, W_\fol) - Rs_p(Y_\gol,W_\fol,X_\fol,Z_\gol) \\
        &\overset{\smash{\mathclap{(\ref{thm:s2w-lagr-symmetries:bianchi})}}}{=}
           Rs_p(X_\fol, Y_\gol, Z_\gol, W_\fol) + Rs_p(Y_\gol,X_\fol,Z_\gol,W_\fol) +
           Rs_p(Y_\gol,Z_\gol,W_\fol,X_\fol) \\
        &\overset{\smash{\mathclap{(\ref{thm:s2w-lagr-symmetries:fols})}}}{=}
           Rs_p(X_\fol, Y_\gol, Z_\gol, W_\fol) + Rs_p(Y_\gol,X_\fol,Z_\gol,W_\fol) \\
        &\overset{\smash{\mathclap{(\ref{thm:s2w-lagr-symmetries:omega})}}}{=}
           2Rs_p(X_\fol, Y_\gol, Z_\gol, W_\fol).
    \end{aligned}
  \end{equation}
  This lead us to $Rs_p(X,Y,Z,W)=0$ for arbitrary $X,Y,Z,W\in T_pM$ by
  multilinearity, since we can decompose each $V\in\smset{X,Y,Z,W}$ into
  $V_\fol+V_\gol$, where $V_\fol\in T_p\fol$ and $V_\gol\in T_p\gol$. Each of
  the 16 resulting terms will vanish due to symmetries of $Rs$ combined with
  the last equality $(\ref{eq:s2w-lagr-sectcurv1})$. Since the choice of the
  point $p\in M$ was arbitrary, the proof is complete. (Lastly, we note that
  this result also follows directly from the bi-Lagrangian/para-Kähler
  correspondence, since the vanishing of the ordinary sectional curvature
  tensor
  \begin{equation}
    K(X_\fol,Y_\gol)=Rm_p(X_\fol,Y_\gol,Y_\gol,X_\fol)
    \overset{(\ref{eq:s2w-lagr-rsrm})}{=}
    \pm Rs_p(X_\fol,Y_\gol,Y_\gol,X_\fol)
  \end{equation}
  for vectors tangent to the foliations $\fol,\gol$ can be easily extended to
  all pairs of vectors $X,Y$ spanning $g$-nondegenerate tangent planes. Having
  this, the classical theory yields the desired result.)
\end{proof}

The actual verification of the above geometric triviality conditions (Theorem
$\ref{thm:s2w-lagr-equiv}$, condition $(\ref{thm:s2w-lagr-main:subwebsother})$)
involves computing the areas of certain curvilinear quadrilaterals lying on
bi-Lagrangian surfaces. While these calculations can be carried out by
integrating a surface volume form induced by the symplectic form, we can
utilize the bi-Lagrangian structure of the ambient space instead to simplify
them significantly.

This simplification depends on the a certain well-known fact regarding the
behavior of a symplectic form $\omega$ with respect to a pair
of complementary Lagrangian foliations $\fol,\gol$. Its statement involves the
graded derivations $d_x,d_y$ of $\Omega^\bullet(\Rb{2n})$ satisfying
$d=d_x+d_y$ which extend the operation of taking the differentials $d_x
f_{|(x,y)} = \smsum{i} \smash{\smpd{f}{x_i}(x,y)}\,dx_i$ and $d_y f_{|(x,y)} =
\smsum{j} \smash{\smpd{f}{y_j}(x,y)}\,dy_j$ of a smooth function
$f\in\Omega^0(\Rb{2n})$ in directions tangent to leaves of a single foliation
$T\fol=\bigcap_{i=1}^n\ker dy_i =
\left\langle\basis{x_1},\basis{x_2},\ldots,\basis{x_n}\right\rangle$ or
$T\gol=\bigcap_{i=1}^n\ker dx_i =
\left\langle\basis{y_1},\basis{y_2},\ldots,\basis{y_n}\right\rangle$
(see the paragraph preceding Proposition $\ref{thm:s2w-lagr-conn-coords}$).
\begin{lem}[\cite{2-webs}]
  \label{thm:s2w-lagr-dp}
  Let $\omega=\smsum{i,j}A_{ij}\,dx_i\wedge dy_j\in\Omega^2(\Rb{2n},0)$ be a
  smooth $2$-form germ satisfying $d\omega=0$, where $A_{ij}\in
  C^\infty(\Rb{2n},0)$. There exists a smooth function-germ $h\in
  C^\infty(\Rb{2n},0)$ for which the following equality holds:
  \begin{equation}
    \omega = d_xd_yh.
  \end{equation}
\end{lem}

For instance, given a bi-Lagrangian structure-germ on the space of rays
$(T^*S^{n-1},\omega,\fol,\gol)$ induced by a pair of hypersurfaces $H,K$
parametrized by $s\mapsto x(s)\in H$ and $t\mapsto y(t)\in K$ (see Example
$\ref{ex:s2w-rays}$) the symplectic form $(\ref{eq:s2w-lagr-rays-omega})$
in coordinates $(s_1,\ldots,s_{n-1},t_1,\ldots,t_{n-1})$ satisfying
$T\fol=\bigcap_{i=1}^{n-1}\ker dt_i$ and $T\gol=\bigcap_{i=1}^{n-1}\ker ds_i$
reduces to
\begin{equation}
  \label{eq:s2w-lagr-rays-h}
  \omega = -d\Big(\sum_{i=1}^n\frac{x_i(s)-y_i(t)}{\abs{x(s)-y(t)}}\,dx_i\Big)
      = -d(d_s\abs{x(s)-y(t)}) = d_sd_t\abs{x(s)-y(t)},
\end{equation}
where in the last equality we used the identites $d=d_s+d_t$ and $d_td_t=0$.
Hence, in this case we can take $h(s,t)=\abs{x(s)-y(t)}$ as the double
potential of $\omega$ inside the statement of Lemma $\ref{thm:s2w-lagr-dp}$.

Now, if a surface $S$ has a boundary composed of four piecewise-smooth curves
$\gamma_1,\gamma_2,\gamma_3,\gamma_4$ such that $\dot\gamma_1,\dot\gamma_3\in
T\fol$, $\dot\gamma_2,\dot\gamma_4\in T\gol$, $\gamma_1(0)=\gamma_4(1)=(x,y)$
$\gamma_1(1)=\gamma_2(0)=(x',y)$, $\gamma_2(1)=\gamma_3(0)=(x',y')$ and
$\gamma_3(1)=\gamma_4(0)=(x,y')$, the integral of $\omega$ over $S$ simplifies
to
\begin{equation}
  \label{eq:s2w-lagr-integral-h}
  \begin{aligned}
    \int_S \omega &= \int_S d_xd_yh = \int_S d(d_yh) = \int_{\partial S} d_y h \\
      &= \underbrace{\int_{\gamma_1} d_y h}_{\mathclap{\dot\gamma_1\in T\fol\subseteq
      \ker d_yh}}
      + \int_{\gamma_2} d_y h
      + \underbrace{\int_{\gamma_3} d_y h}_{\mathclap{\dot\gamma_3\in T\fol\subseteq
      \ker d_yh}}
      + \int_{\gamma_4} d_y h \\
      &= \mask{\int_{\gamma_1} d_y h}{0}
      + \int_{\gamma_2} d_y h
      + \mask{\int_{\gamma_3} d_y h}{0}
      + \int_{\gamma_4} d_y h \\
      &= h(\gamma_2(1)) + h(\gamma_4(1)) - h(\gamma_2(0)) - h(\gamma_4(0)) \\
      &= h(x',y') + h(x,y) - h(x',y) - h(x,y').
  \end{aligned}
\end{equation}

\begin{figure}
  \centering
  \def\svgwidth{\textwidth}
  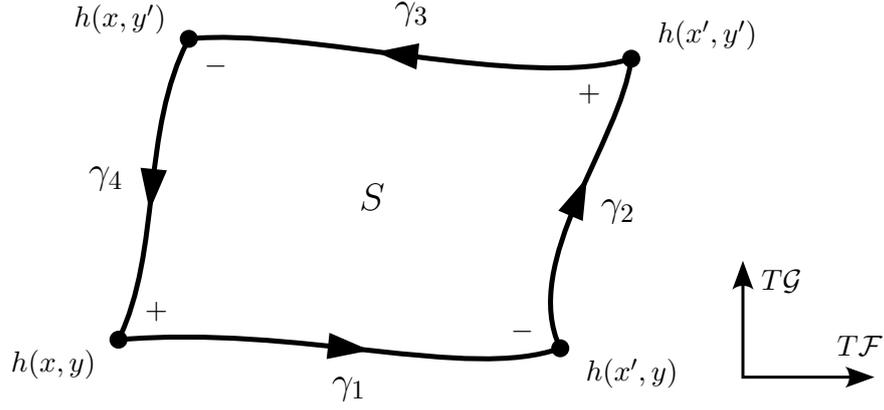
  \caption{\footnotesize Visualisation of integral formula
  $(\ref{eq:s2w-lagr-integral-h})$ in terms of the double potential $h(x,y)$ of
  the $2$-form $\omega$.}
  \label{fig:s2w-lagr-integral-dp}
\end{figure}

Integrals of these kind provide a foundation for a more refined geometric
interpretation of the curvature of bi-Lagrangian manifolds. We can use them to
give several conditions for flatness of the bi-Lagrangian structure $(M,
\omega, \fol, \gol)$ in terms of the function $h(x,y)$ of Lemma
$\ref{thm:s2w-lagr-dp}$ and Hamiltonians $f,g$ satisfying $df_{|T\fol}=0$ and
$dg_{|T\gol}=0$, which involve signed values of $h(x,y)$ on vertices of certain
quadrilaterals spanned by the integral curves $\gamma_x,\gamma_y$ of the
corresponding Hamiltonian vector fields $X_f,X_g\in\mathfrak{X}(M)$ (see Figure
$\ref{fig:s2w-lagr-integral-dp}$). As an example, we formulate three of these
flatness conditions by carrying over the statement of Theorems
$\ref{thm:s2w-lagr-equiv}$ and $\ref{thm:dfw-geom}$ directly into our setting.

\begin{thm}
  \label{thm:s2w-lagr-geom-h}
  Let $(M,\omega,\fol,\gol)$ be a bi-Lagrangian structure. For a given pair of
  leaves $F\in\fol$ and $G\in\gol$ let $(x_1,\ldots,x_n,y_1,\ldots,y_n)$ be a local
  coordinate system satisfying $F\cap G=\smset{0}$, $T\fol=\bigcap_{i=1}^n\ker
  dy_i$ and $T\gol=\bigcap_{i=1}^n\ker dx_i$, so that the coordinates $(x_1,\ldots,x_n)$
  parametrize $F$, the coordinates $(y_1,\ldots,y_n)$ parametrize $G$, and each
  point $p=(x,y)\in M$ corrseponds bijectively to a pair of points
  $p_x=(x_1,\ldots,x_n)\in F$ and $p_y=(y_1,\ldots,y_n)\in G$.

  \medskip\noindent
  Then the bi-Lagrangian connection $\nabla$ of $(M, \omega,\fol,\gol)$ is flat
  if and only if either of the conditions below is true for each pair of leaves
  $F\in\fol,G\in\gol$ and each pair of integral curves $\gamma_x\subseteq
  F,\gamma_y\subseteq G$ of Hamiltonian vector fields
  $X_f,X_g\in\mathfrak{X}(M)$ corresponding to Hamiltonians $f, g\in
  C^\infty(M)$ satisfying $df_{|T\fol}=0$ and $dg_{|T\gol}=0$.
  \begin{enumerate}[label=$(\arabic{enumi})$, ref=\arabic{enumi}]
    \item For each quadruple of points $p_1,p_3\in\gamma_x$, $q_1,q_3\in\gamma_y$ there
      exist two points $p_2\in\gamma_x$, $q_2\in\gamma_y$ such that
      \begin{equation}
        \begin{aligned}
          &h(p_1,q_1) + h(p_2,q_2) - h(p_1,q_2) - h(p_2,q_1) \\
          &\hskip 1em = h(p_1,q_2) + h(p_2,q_3) - h(p_1,q_3) - h(p_2,q_2) \\
          &\hskip 1em = h(p_2,q_2) + h(p_3,q_3) - h(p_2,q_3) - h(p_3,q_2) \\
          &\hskip 1em = h(p_2,q_1) + h(p_3,q_2) - h(p_2,q_2) - h(p_3,q_3).
        \end{aligned}
      \end{equation}
    \item For each triple of points $p_1,p_2,p_3\in\gamma_x$ and each triple
      $q_1,q_2,q_3\in\gamma_y$ for which it holds that
      \begin{equation}
        \begin{aligned}
          &h(p_1,q_1) + h(p_3,q_2) - h(p_1,q_2) - h(p_3,q_1) \\
          &\hskip 1em = h(p_1,q_2) + h(p_3,q_3) - h(p_1,q_3) - h(p_3,q_2)
        \end{aligned}
      \end{equation}
      the equality
      \begin{equation}
        \begin{aligned}
          &h(p_1,q_1) + h(p_2,q_2) - h(p_1,q_2) - h(p_2,q_1) \\
          &\hskip 1em = h(p_1,q_2) + h(p_2,q_3) - h(p_1,q_3) - h(p_2,q_2)
        \end{aligned}
      \end{equation}
      implies
      \begin{equation}
        \begin{aligned}
          &h(p_1,q_1) + h(p_2,q_2) - h(p_1,q_2) - h(p_2,q_1) \\
          &\hskip 1em = h(p_1,q_2) + h(p_2,q_3) - h(p_1,q_3) - h(p_2,q_2) \\
          &\hskip 1em = h(p_2,q_2) + h(p_3,q_3) - h(p_2,q_3) - h(p_3,q_2) \\
          &\hskip 1em = h(p_2,q_1) + h(p_3,q_2) - h(p_2,q_2) - h(p_3,q_3).
        \end{aligned}
      \end{equation}
    \item For each triple of points $p_1,p_2,p_3\in\gamma_x$ and each triple
      $q_1,q_2,q_3\in\gamma_y$ the following equality is satisfied.
      \begin{equation}
        \begin{aligned}
          &\big(
            h(p_1,q_1) + h(p_2,q_2) - h(p_1,q_2) - h(p_2,q_1)\big)\\
          &\hskip 3em
            \cdot \big(h(p_2,q_2) + h(p_3,q_3) - h(p_2,q_3) - h(p_3,q_2)\big) \\
          &\hskip 1em
            = \big(h(p_1,q_2) + h(p_2,q_3) - h(p_1,q_3) -
          h(p_2,q_2)\big)\\
          &\hskip 3em
            \cdot\big(h(p_2,q_1) + h(p_3,q_2) - h(p_2,q_2) - h(p_3,q_3)
            \big).
        \end{aligned}
      \end{equation}
  \end{enumerate}
\end{thm}
\begin{proof}
  For each pair of integral curves $\gamma_x\in F,\gamma_y\in G$ of
  $X_f,X_g\in\mathfrak{X}(M)$ one can find a bi-Lagrangian surface generated by
  Hamiltonians $f,g\in C^\infty(M)$ at $p\in F\cap G$ which is spanned by
  $\gamma_x,\gamma_y$ by Lemma $\ref{thm:s2w-lagr-surf}$. With this in mind,
  apply Lemma $\ref{thm:s2w-lagr-dp}$ and formula
  $(\ref{eq:s2w-lagr-integral-h})$ to Theorem $\ref{thm:s2w-lagr-equiv}$.
\end{proof}

\begin{figure}
  \centering
  \def\svgwidth{\textwidth}
  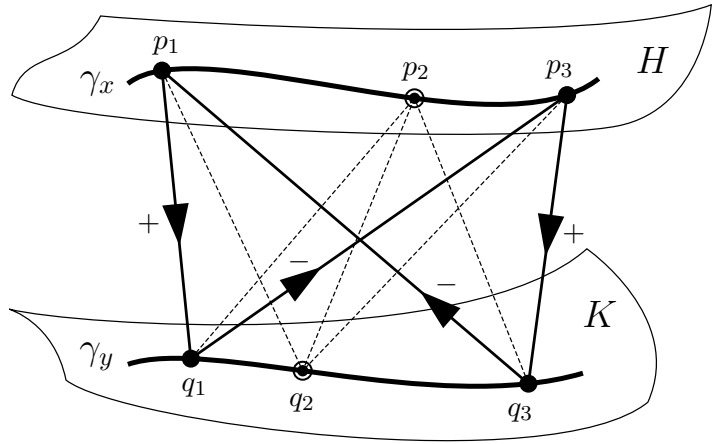
  \caption{\footnotesize Geometric interpretation of the symplectic form
  $\omega$ on the space of rays in terms of the embedding of hypersurfaces
  $H\times K$ into $T^*S^{n-1}$. To obtain the symplectic area of the
  $2$-dimensional region bounded by leaves of $\fol,\gol$ crossing $(p_1,q_1)$
  and $(p_3,q_3)$ spanned by $\gamma_x,\gamma_y$, traverse the piecewise-linear
  path $p_1q_1p_3q_3p_1$ connecting $H,K$ and add the distances between
  succesive points with positive sign when reaching $K$ from $H$ and with
  negative sign when travelling the other way. See
  $(\ref{eq:s2w-lagr-rays-h})$, $(\ref{eq:s2w-lagr-integral-h})$ and Theorem
  $\ref{thm:s2w-lagr-geom-h}$.}
  \label{fig:s2w-rays-flatness}
\end{figure}

\section{Two problems on existence of flat bi-Lagrangian structures}
\label{ch:s2w-lagr-problems}

\subsection{Bi-Lagrangian flatness of the space of rays}
\label{ch:s2w-lagr-rays}

This section is a continuation of Example $\ref{ex:s2w-rays}$, which is
concerned with a certain bi-Lagrangian structure $(U,\omega,\fol,\gol)$ over an
open subset $U$ of the ray space $(T^*S^{n-1}, \omega)$, i.e. the symplectic
space of oriented affine lines in $\Rb{n}$, in which each leaf of the two
Lagrangian foliations $\fol,\gol$ is composed of rays passing through a single
point $x$ lying on one of two fixed disjoint hypersurfaces $H,K$. In his
publication \cite[{}I.6]{2-webs} Tabachnikov posed the following question:
\emph{for which hypersurfaces $H,K$ the bi-Lagrangian structure induced on
$T^*S^{n-1}$ by $H,K$ is flat?}

The answer to this question is reachable through raw calculation, yielding
Theorem $\ref{thm:s2w-lagr-rays-flat}$ as a result. However, due to the sheer
difficulty of computations involving the bi-Lagrangian curvature tensor
$(\ref{eq:s2w-lagr-conn-coords-curv})$ for general dimension $n$, it seems
necessary to take an indirect route and split the solution into two parts. The
first part is to solve a simplified problem: \emph{find hypersurface-germs
$H,K$ for which the bi-Lagrangian structure is \emph{Ricci}-flat}, while the
goal of the second part is to eliminate those among Ricci-flat bi-Lagrangian
structures which are not flat. Since a non Ricci-flat structure has non-zero
curvature, the above reasoning suffices to settle the problem. Both parts are
computationally intense in their own way. It is well-advised to verify the
following results with the help of a computer algebra system. The authors
themselves were assisted by \texttt{Wolfram Mathematica 13} \cite{math} in the
process of deriving the next few theorems.

Before giving the statements of the main theorems of this section, let us
simplify the problem by putting the bi-Lagrangian structure of the space of
rays into a more calculation-friendly normal form at generic points $(p,q)\in
U\subseteq T^*S^{n-1}$ of its domain.

The genericity condition in question is: \emph{a ray $(p,q)\in U$ intersects
$H$ and $K$ transversely}. It is not difficult to prove that, given a
bi-Lagangian structure $(U,\omega,\fol,\gol)$ of this kind, the set of rays
satisfying this condition is indeed open and dense in $U$. Let us denote by
$\mathrm{Tv}_H, \mathrm{Tv}_K$ the sets of rays in $U$ having transverse
intersection with $H$ and $K$ respectively. They are easily seen to be open in
$U$. To see that they are dense in $U$, assume that $(p,q)\not\in
\mathrm{Tv}_H$. Then, the ray $\ell$ represented by $(p,q)$ is contained in
some $T_xH$ for $x\in H$. Since $U$ is open, we can rotate the ray $\ell$ by a
small angle about the point $x\in H$ in arbitrary direction, which can be
chosen in such a way that the new ray $\ell'$ with parameters $(p',q')\in U$ is
no longer contained in $T_xH$. Hence, $\ell'$ intersects $T_xH$ at a single
point, namely $x\in H$, and therefore is transverse to $H$. Moreover, since
$(p',q')$ lies in $U$, the ray $\ell'$ intersects both $H$ and $K$ by
definition of the bi-Lagrangian structure. Thus, $\mathrm{Tv}_H$ is open and
dense in $U$, and so is $\mathrm{Tv}_K$ by the same argument. Since an
intersection of two open and dense subsets is open and dense, the rays with the
joint transversality property above are indeed generic.

\begin{lem}
  \label{thm:s2w-lagr-rays-norm}
  Let $H,K$ be two disjoint hypersurfaces in $\Rb{n}$, and let $r\in
  T^*S^{n-1}$ be a ray intersecting both hypersurfaces transversely at
  points $x_0\in H$, $y_0\in K$. There exists an orthonormal coordinate system
  $(x_1,\ldots,x_{n},y_1,\ldots,y_{n})$ of $\Rb{2n}$ in which:
  \begin{enumerate}[label=$(\alph{enumi})$, ref=\alph{enumi}]
    \item $r=(p,q)$, where $q=0$ and $p=(0,0,\ldots,0,1)\in\Rb{n}$,
    \item the germs of $H,K$ at $x_0, y_0$ are graphs of smooth function germs
      $f,g\in C^\infty(\Rb{n-1},0)$ respectively, meaning that the pairs of
      points $x\in H$ and $y\in K$ are parametrized by
      \begin{equation}
        x(s) = (s_1,\ldots,s_{n-1},f(s)),\qquad
        y(t) = (t_1,\ldots,t_{n-1},g(t))
      \end{equation}
      for $s=(s_1,\ldots,s_{n-1})\in\Rb{n-1}$ and
      $t=(t_1,\ldots,t_{n-1})\in\Rb{n-1}$,
    \item the parameters $s,t\in\Rb{n-1}$ form a local coordinate system on
      $T^*S^{n-1}$ by means of a map $(s,t)\mapsto (x(s),y(t))\mapsto r(s,t)$,
      where $r(s,t)\in T^*S^{n-1}$ is the unique ray passing through $x(s)$ and
      $y(s)$,
    \item the symplectic form $\omega$ on $T^*S^{n-1}$ is
      \begin{equation}
        \label{eq:s2w-lagr-rays-norm:omega}
        \begin{aligned}
          \omega &= \sum_{i,j=1}^{n-1}\frac{
            \big(s_i-t_i+\smpd{f}{s_i}(s)\cdot (f(s)-g(t))\big)
            \big(s_j-t_j+\smpd{g}{t_j}(t)\cdot (f(s)-g(t))\big)}
            {\big(\sum_{k=1}^{n-1}(s_k-t_k)^2+(f(s)-g(t))^2\big)^{3/2}}\,ds_i\wedge dt_j
            \\[1ex]
          &\hskip 3em -
          \sum_{i,j=1}^{n-1}\frac{\smpd{f}{s_i}(s)\smpd{g}{t_j}(t)+\delta_{ij}}
            {\big(\sum_{k=1}^{n-1}(s_k-t_k)^2+(f(s)-g(t))^2\big)^{1/2}}\,ds_i\wedge
            dt_j,
        \end{aligned}
      \end{equation}
      where $\delta_{ij}$ is the Kronecker's delta,
    \item the volume form on $T^*S^{n-1}$ induced by $\omega$ is
      $\omega^{n-1} = h(s,t)\,ds_1\wedge\cdots\wedge ds_{n-1}\wedge
      dt_1\wedge\cdots\wedge dt_{n-1}$, where
      \begin{equation}
        \label{eq:s2w-lagr-rays-norm:vol}
        \begin{aligned}
          &h(s,t) =
            \frac{(-1)^{\binom{n}{2}}(n-1)!
              \big(f-g-\smsum{k}\smpd{f}{s_k}\cdot(s_k-t_k)\big)
              \big(f-g-\smsum{k}\smpd{g}{t_k}\cdot(s_k-t_k)\big)}
              {\big(\sum_{k=1}^{n-1}(s_k-t_k)^2+(f-g)^2\big)^{\frac{n+1}{2}}}.
        \end{aligned}
      \end{equation}
  \end{enumerate}
\end{lem}
\begin{proof}
  Using rigid motions $R\in\mathrm{SO}(n,\Rb{1})$ we can arrange the initial
  coordinates $(x_1,\ldots,x_n)$ in such a way that the ray which connects the
  two given points on hypersurfaces $H,K$ has parameters $p_0=(0,0,\ldots,0,1)$
  and $q_0=(0,0,\ldots,0,0)$. In these coordinates $T_{x_0}H\oplus\langle
  p_0\rangle = T_{x_0}\Rb{n}$ and $T_{y_0}K\oplus\langle p_0\rangle =
  T_{y_0}\Rb{n}$, hence the projections from $H$ and $K$ onto the first $n-1$
  coordinates is a local diffeomorphism. This proves that the germs of
  hypersurfaces $H,K$ are given by the graphs of smooth function-germs
  $\maps{f,g}{(\Rb{n-1},0)}{\Rb{1}}$ respectively, and, according to Example
  $\ref{ex:s2w-rays}$, the space of rays intersecting both hypersurfaces $M$
  can be parametrized by pairs of points $x\in H,y\in K$, each dependent on the
  set of $n-1$ parameters $(s_1,\ldots,s_{n-1})$ and $(t_1,\ldots,t_{n-1})$
  with $x_i = s_i, y_j = t_j$ for $i,j=1,\ldots,n-1$ and $x_n=f(s),y_n=g(t)$.
  The formula $(\ref{eq:s2w-lagr-ex-A})$ for the symplecitc form $\omega$ on $M$
  expands to $(\ref{eq:s2w-lagr-rays-norm:omega})$, while the induced volume
  form $\omega^{n-1}$ given by $(\ref{eq:s2w-lagr-rays-vol})$ becomes exactly
  $(\ref{eq:s2w-lagr-rays-norm:vol})$.
\end{proof}

Now, our global problem reduces to the following local one: \emph{find two
function-germs $f,g\in C^\infty(\Rb{n-1},0)$ that satisfy a set of partial
differential equations expressing the vainshing of the bi-Lagrangian (Ricci)
curvature associated with bi-Lagrangian structure
$(\Rb{2n-2},\omega,\fol,\gol)$ in standard coordinates
$(s_1,\ldots,s_{n-1},t_1,\ldots,t_{n-1})$, where $T\fol = \bigcap_{i=1}^{n-1}
\ker dt_i$, $T\gol = \bigcap_{j=1}^{n-1} \ker ds_j$ and $\omega$ is given by
$(\ref{eq:s2w-lagr-rays-norm:omega})$.} It is known (Proposition
$\ref{thm:dfw-ricci}$, see also \cite{dfw, vaisman}) that the Ricci tensor of
the bi-Lagrangian connection $\nabla$ is exactly
\begin{equation}
  \label{eq:s2w-lagr-rays:ricci}
  \mathrm{Rc} = \sum_{i,j=1}^{n-1}
  \kappa_{ij}(s,t)\,ds_idt_j,\quad\text{where}\quad\kappa_{ij} =
  \pdd{\log\abs{h}}{s_i}{t_j},\quad i,j=1,2,\ldots,n-1,
\end{equation}
for the smooth function $h(s,t)$ given in
$(\ref{eq:s2w-lagr-rays-norm:vol})$. The resulting system of PDEs is severely
overdetermined for large $n$, with $(n-1)^2$ generically independent equations
constraining $2(n-1)$ variables. Our intuition might suggest that in this case
solutions to this system should not exist for almost all $n$. Indeed, a formal
reasoning leads exactly to this conclusion.

\begin{thm}
  \label{thm:s2w-lagr-rays-main}
  Let $n\neq 3$ be a natural number. For any two disjoint hypersurfaces $H,K$
  in $\Rb{n}$ and any ray $(p_0,q_0)\in T^*S^{n-1}$ intersecting both
  hypersurfaces transversely at points $x_0\in H$, $y_0\in K$, the canonical
  connection $\nabla$ associated with the germ at $(p_0,q_0)$ of the
  bi-Lagrangian structure on $T^*S^{n-1}$ induced by the germs of $H,K$ at
  $x_0,y_0$ respectively is not Ricci flat.
\end{thm}

\begin{proof}
  Let $\nabla$ be the bi-Lagrangian connection on the space of rays. Find an
  orthogonal coordinate system $(x_1,\ldots,x_n)$ such that the hypersurfaces
  $H,K$ are given by the graphs of smooth function-germs
  $\maps{f,g}{(\Rb{n-1},0)}{\Rb{1}}$ parametrized by
  $(s_1,\ldots,s_{n-1})\in\Rb{n-1}$ and $(t_1,\ldots,t_{n-1})\in\Rb{n-1}$ as in
  Lemma $\ref{thm:s2w-lagr-rays-norm}$. Denote the coefficients of its Ricci
  tensor $\mathrm{Rc}=\sum_{i,j=1}^{n-1} \kappa_{ij}(s,t)\,ds_idt_j$
  $(\ref{eq:s2w-lagr-rays:ricci})$ by $\kappa_{ij}\in C^\infty(\Rb{2n-2},0)$.
  Next, let $\kappa_{ij}(s,t)$ be as in $(\ref{eq:s2w-lagr-rays:ricci})$, fix
  $i=1,2,\ldots,n-1$ and let
  \begin{equation}
    \tilde{c}_{jk}(s)=\frac{\partial^{j+k}\kappa_{ii}}{\partial s_i^j\partial
    t_i^k}_{|(s,s)}
  \end{equation}
  for each $j,k=1,2,\ldots,n-1$. Express these quantities with the help of
  auxiliary functions
  \begin{equation}
    \begin{aligned}
      \sigma(s)&=f(s)+g(s),&\qquad
      \sigma_k(s) &= \smfrac{\partial^k\sigma}{\partial s_i^k}(s),\\
      \rho(s)&=f(s)-g(s),&\qquad
      \rho_k(s) &= \smfrac{\partial^k\rho}{\partial s_i^k}(s).
    \end{aligned}
  \end{equation}
  It turns out that that $\tilde{c}_{jk}$ for $j,k=0,1,2,3$ are rational
  functions of $\rho,\sigma$ and their derivatives $\rho_m,\sigma_m$ of order
  $m\in\mathbb{N}$, the denominators of which equal to $\rho^{j+k+2}$. The
  function $\rho$ is nonvanishing since the two hypersurface-germs $H,K$ do not
  intersect. Let
  \begin{equation}
    c_{jk} = \rho^{j+k+2}\tilde{c}_{jk} \quad \text{for }j,k\in\mathbb{N}.
  \end{equation}
  To prove the theorem it is enough to show that a
  pair of functions $(\sigma,\rho)$ with $\rho$ nonvanishing which satisfies
  $c_{jk}=0$ for each $j,k\in\mathbb{N}$ does not exist. The first of these
  equations is
  \begin{equation}
    0 = c_{00} = (1+n)(4+\rho_1^2-\sigma_1^2)+4\rho\rho_2.
  \end{equation}
  Note that it allows us to determine $\rho_2$ in terms of $\rho,\rho_1$ and
  $\sigma_1$, namely
  \begin{equation}
    \rho_2 = \frac{(1+n)(\sigma_1^2-\rho_1^2 - 4)}{4\rho}
  \end{equation}
  We arrive at a similar situation if we examine equations of the form
  $c_{jk}+c_{kj}=0$ and $c_{jk}-c_{kj}=0$ for pairs of positive indices $(j,k)$
  satisfying $j\cdot k = 0$ and $\max(k,l)\leq 3$. These equations again let us
  write the derivatives $\rho_k,\sigma_k$ of $\rho$ and $\sigma$ for $3\leq
  k\leq 5$ in terms of the lower ones and, ultimately, in terms of
  $\rho,\rho_1,\sigma_1$ and $\sigma_2$. Even in their fully reduced form they
  may look quite intimidating, hence we state only the first three for brevity.
  \begin{align}
    0 = c_{10}+c_{01}
      &= 4\rho^2\rho_3-\smfrac12(1+n)\Big(
        (3+n)\rho_1(4+\rho_1^2-\sigma_1^2) + 4\rho\sigma_1\sigma_2\Big),\\
    0 = c_{10}-c_{01}
      &= \smfrac12\Big(
        (1+n)\big(4(n-11)+(n-3)\rho_1^2\big)\sigma_1\notag \\
      &\hskip 3em -(n-3)(n+1)\sigma_1^3 +
        4\rho\big((n+7)\rho_1\sigma_2+2\rho\sigma_3\big)\Big),\\
    0 = c_{20}+c_{02}
      &= \smfrac{1}{2} \Big(2 (n+1) \rho_1^2 \big(2 n (n+20)+6-(11 n+3)
      \sigma_1^2\big)\notag\\
      &\hskip 3em {}+8 (n+1) (n+6) \rho \rho_1 \sigma_1
      \sigma_2+(n+1) (n (n+14)+9) \rho_1^4\notag\\
      &\hskip 3em {}+(n+1) \Big(4 ((n-22) n+49) \sigma_1^2+96 (n-1)-((n-8) n+3)
      \sigma_1^4\Big)\notag\\
      &\hskip 3em {}+16 \rho^2 \big(\rho
      \rho_4+3 \sigma_2^2\big)\Big).
  \end{align}
  Note that the term $\sigma_2^2$ can also be reduced to an expression
  involving only $\rho,\rho_1,\sigma_1$ by means of the equation
  \begin{equation}
    \label{eq:s2w-lagr-rays-main:s22}
    \begin{aligned}
      0 = c_{11} &=
        \smfrac{(n+1)(n-3)(n-5)}8\Big(\rho_1^4-8\sigma_1^2+\sigma_1^4-2\rho_1^2(\sigma_1^2-4)+16\Big)\\
      &\hskip 3em -2\Big(24(n+1)\sigma_1^2+(n+7)\rho^2\sigma_2^2\Big).
    \end{aligned}
  \end{equation}
  With the help of these eight equations we bring all expressions $c_{jk}$ with
  $j,k=0,1,2,3$ to polynomials in variables $\rho,\rho_1,\sigma_1$ and
  $\sigma_2$ with coefficients depending on the dimension $n$, where
  $\sigma_2$ occurs only in its first power.

  To progress further we need several auxiliary lemmas.
  \begin{enumerate}[label=$(\arabic{enumi})$, ref=\arabic{enumi}]
    \item\label{thm:s2w-lagr-rays-main:lem-rs0}
      \emph{If $n\neq 3$, then for every $s\in\Rb{n-1}$ near the origin the
      values $\sigma_1(s)$ and $\rho_1(s)$ cannot both be zero.}
  \end{enumerate}
  \emph{Proof.}\hskip 1em
  Assume the contrary, that for some $s\in\Rb{n-1}$ we have
  $\sigma_1=\rho_1=0$. Then equalities $c_{11}=0$ and $c_{22}=0$ at
  $s\in\Rb{n-1}$ give
  \begin{equation}
    \begin{aligned}
      0 &= -(n+7) \rho^2 \sigma_2^2+n^3-7 n^2+7 n+15, \\
      0 &= -9 (n+7) \rho^2 \sigma_2^2+ n^3+ 21 n^2-169 n+291
    \end{aligned}
  \end{equation}
  respectively. Combining them yields a polynomial equation in $n$
  \begin{equation}
    2n^3-21n^2+58n-39=0
  \end{equation}
  with roots $n=1,3,\smfrac{13}{2}$. Since $n\in\mathbb{N}$, $n\geq 2$ and
  $n\neq 3$ by assumption, we arrive at a contradiction.
  Therefore $(\rho_1(s),\sigma_1(s))\neq (0,0)$ at each $s\in\Rb{n-1}$ in the
  domain of $\rho,\sigma$.
  \begin{enumerate}[label=$(\arabic{enumi})$, ref=\arabic{enumi}]
    \setcounter{enumi}{1}
      \item\label{thm:s2w-lagr-rays-main:lem-r0}
      \emph{If $n\neq 3$, then the function $\rho_1$ is nonvanishing.}
  \end{enumerate}
  \emph{Proof.}\hskip 1em
  Assume that $\rho_1=0$ at some point $s\in\Rb{n-1}$. Then $c_{12}-c_{21}=0$
  at $s\in\Rb{n-1}$ gives
  \begin{equation}
    0 = \big(n^2-8 n+15\big) \sigma_1^5-16 \big(n^2-5 n+54\big) \sigma_1^3
      +48 \big(n^2-12n+27\big) \sigma_1.
  \end{equation}
  Hence either $\sigma_1=0$, by which we obtain $\rho_1=\sigma_1=0$
  contradicting Lemma $(\ref{thm:s2w-lagr-rays-main:lem-rs0})$, or
  \begin{equation}
    \label{eq:s2w-lagr-rays-main:lr0-1}
    0 = \big(n^2-8 n+15\big) \sigma_1^4-16 \big(n^2-5 n+54\big) \sigma_1^2
      +48 \big(n^2-12n+27\big).
  \end{equation}
  Now, take $c_{22}=0$. This yields
  \begin{equation}
    \label{eq:s2w-lagr-rays-main:lr0-2}
    \begin{aligned}
      0&=\big(n^4-11 n^3+68 n^2-169 n+111\big) \sigma_1^6
        + 12 \big(16 n^3-41 n^2+500 n-1243\big) \sigma_1^4 \\
        &{}- 16 \big(n^4+23 n^3-186 n^2+1525 n-3667\big) \sigma_1^2
        + 192 \big(2 n^3-21 n^2+58 n-39\big).
    \end{aligned}
  \end{equation}
  It can be verified directly that equations
  $(\ref{eq:s2w-lagr-rays-main:lr0-1})$ and
  $(\ref{eq:s2w-lagr-rays-main:lr0-2})$ have no common zeroes, a contradiction.
  This proves that $\rho_1\neq 0$ at each $s\in\Rb{n-1}$.
  \begin{enumerate}[label=$(\arabic{enumi})$, ref=\arabic{enumi}]
    \setcounter{enumi}{2}
    \item\label{thm:s2w-lagr-rays-main:lem-s0}
      \emph{If $n\neq 3,5$, then the function $\sigma_1$ is nonvanishing.}
  \end{enumerate}
  \emph{Proof.}\hskip 1em
  Assume that $\sigma_1=0$ at some point $s\in\Rb{n-1}$. The equation
  $c_{12}+c_{21}=0$ at $s\in\Rb{n-1}$ is equivalent to
  \begin{equation}
    0 = (n-5) (n-3) \rho_1 (\rho_1^2+4)^2.
  \end{equation}
  Since $n\neq 3,5$ by assumption, we get $\rho_1=0$ at $s\in\Rb{n-1}$, which
  leads to a contradiction with Lemma $(\ref{thm:s2w-lagr-rays-main:lem-r0})$.
  Hence $\sigma_1\neq 0$ everywhere.
  \begin{enumerate}[label=$(\arabic{enumi})$, ref=\arabic{enumi}]
    \setcounter{enumi}{3}
    \item\label{thm:s2w-lagr-rays-main:lem-s20}
      \emph{If $n\neq 3,5$, then the function $\sigma_2$ is nonvanishing.}
  \end{enumerate}
  \emph{Proof.}\hskip 1em
  Assume that $\sigma_2=0$ at some $s\in\Rb{n-1}$. Equation $c_{12}+c_{21}=0$
  expands to
  \begin{equation}
    0 = \rho_1(n+1)\Big(48 \sigma_1^2-\smfrac{1}{16} (n-5)
        (n-3) (n+3) \big(-2 \rho_1^2 (\sigma_1^2-4)+\rho_1^4+\sigma_1^4-8
        \sigma_1^2+16\big)\Big).
  \end{equation}
  Since $\rho_1\neq 0$ by Lemma $(\ref{thm:s2w-lagr-rays-main:lem-r0})$, we can
  divide both sides by $\rho_1(n+1)$ to arrive at
  \begin{equation}
    768 \sigma_1^2 = (n-5)
        (n-3) (n+3) \big(-2 \rho_1^2 (\sigma_1^2-4)+\rho_1^4+\sigma_1^4-8
        \sigma_1^2+16\big).
  \end{equation}
  On the other hand, the equality $c_{11}=0$ gives
  \begin{equation}
    384 \sigma_1^2 = (n-5) (n-3) \big(-2 \rho_1^2 (\sigma_1^2-4)+\rho_1^4
      + \sigma_1^4-8 \sigma_1^2+16\big).
  \end{equation}
  Combining the two we get $(n+3)\sigma_1^2 = 2 \sigma_1^2$, which due to
  $n\geq 2$ implies $\sigma_1=0$. This cannot happen by Lemma
  $(\ref{thm:s2w-lagr-rays-main:lem-s0})$. The above contradiction proves that
  $\sigma_2$ does not vanish.

  \vspace*{1em}
  We now proceed to prove the theorem in cases $n\neq 3,5$. Note that by
  eliminating variables $\rho_3,\sigma_3,\sigma_2^2$ from the equations $c_{12}=0$
  and $c_{21}=0$ we can reduce them to the form
  \begin{align}
    \label{eq:s2w-lagr-rays-main:ab12}a_{12} + b_{12}\cdot \rho\sigma_2 &= 0, \\
    \label{eq:s2w-lagr-rays-main:ab21}a_{21} + b_{21}\cdot \rho\sigma_2 &= 0
  \end{align}
  respectively, where $a_{12},a_{21},b_{12},b_{21}$ are polynomials in
  $\sigma_1,\rho_1$ with coefficients depending on the dimension $n$.
  Multiply both sides of $(\ref{eq:s2w-lagr-rays-main:ab12})$ by $b_{21}$
  and of $(\ref{eq:s2w-lagr-rays-main:ab21})$ by $b_{12}$. By taking the
  difference of the resulting expressions we arrive at the equality
  \begin{equation}
    \label{eq:s2w-lagr-rays-main:eq1}
    a_{12}b_{21}-a_{21}b_{12}=0
  \end{equation}
  which holds at each point $s\in\Rb{n-1}$ inside the domain of
  $\sigma_1,\rho_1$.

  Now, if we multiply both sides of $(\ref{eq:s2w-lagr-rays-main:ab12})$ by
  $a_{12}-b_{12}\cdot\rho\sigma_2$ and $(\ref{eq:s2w-lagr-rays-main:ab21})$ by
  $a_{21}-b_{21}\cdot\rho\sigma_2$, we will obtain
  \begin{equation}
    (a_{12}^2-a_{21}^2) - (b_{12}^2-b_{21}^2)\cdot\rho^2\sigma_2^2 = 0
  \end{equation}
  by again taking the difference of the results. After using
  $(\ref{eq:s2w-lagr-rays-main:s22})$ to express $\rho^2\sigma_2^2$ in terms of
  $n,\sigma_1,\rho_1$ we arrive at an equivalent equality
  \begin{equation}
    \label{eq:s2w-lagr-rays-main:eq2}
    \begin{aligned}
      &(a_{12}^2-a_{21}^2) - (b_{12}^2-b_{21}^2)\cdot \smfrac{n+1}{16(n+7)}
        \Big((n-5) (n-3) \\
      &\hskip 4em \cdot\big(-2 \rho_1^2 (\sigma_1^2-4)+\rho_1^4+\sigma_1^4-8
        \sigma_1^2+16\big)-384 \sigma_1^2\Big) = 0.
    \end{aligned}
  \end{equation}

  Both equations $(\ref{eq:s2w-lagr-rays-main:eq1})$ and
  $(\ref{eq:s2w-lagr-rays-main:eq2})$ involve polynomials in two variables
  $\sigma_1,\rho_1$ with coefficients depending on the dimension $n$. It can be
  verified directly (although this can be infeasible to do by hand) that the
  system of these two equations has only a finite number of solutions for a
  fixed dimension $n\geq 2$. In particular, the values of $\sigma_1$ at each
  point $s\in\Rb{n-1}$ in its domain belong to a finite set. Due to assumed
  smoothness of $\sigma_1$, this function has to be constant by continuity,
  hence $\sigma_2=0$. In dimensions $n\neq 3,5$ this leads to a contradiction
  with Lemma $(\ref{thm:s2w-lagr-rays-main:lem-s20})$, proving that the Ricci
  curvature cannot be null.

  \vspace*{1em}
  Now, let us consider the remaining case $n=5$. In the following we will make
  use of notations introduced at the beginning of the proof. Consider
  $c_{00}=0$, namely
  \begin{equation}
    \label{eq:s2w-lagr-rays-main:n5c00}
    0=12+3\rho_1^2+2\rho\rho_2-3\sigma_1^2,
  \end{equation}
  and differentiate its both sides with respect to $s_i$ to obtain
  \begin{equation}
    0 = 4\rho_1\rho_2+\rho\rho_3-3\sigma_1\sigma_2.
  \end{equation}
  Take the equality $c_{11}=0$, which expands to
  \begin{equation}
    \begin{aligned}
      0 &= -12 \rho^2 \sigma_2^2-2 \rho_1^2 \big(10 \rho \rho_2+9
      (\sigma_1^2-4)\big)-8 \rho \rho_1 \big(\rho \rho_3-3 \sigma_1
      \sigma_2\big) \\
      &\hskip 8em {}+\big(2 \rho \rho_2-3 \sigma_1^2\big)^2+24 \big(2 \rho
      \rho_2-9 \sigma_1^2+6\big)+9 \rho_1^4,
    \end{aligned}
  \end{equation}
  and note that if we let $q(s)=c_{11}(s)$, $p_1(s) = c_{00}(s)$ and
  $p_2(s)=(\basis{s_i}c_{00})(s)$, then
  \begin{equation}
    q(s) = (60+3\rho_1^2+2\rho\rho_2-3\sigma_1^2)\cdot p_1(s) -
    (8\rho\rho_1)\cdot p_2(s)
    -12(48+12\rho_1^2+8\rho\rho_2+\rho^2\sigma_2^2).
  \end{equation}
  Hence,
  \begin{equation}
    48+12\rho_1^2+8\rho\rho_2+\rho^2\sigma_2^2 = 0,
  \end{equation}
  which together with $(\ref{eq:s2w-lagr-rays-main:n5c00})$ gives
  \begin{equation}
    12\sigma_1^2+\rho^2\sigma_2^2=0.
  \end{equation}
  The equality says that a sum of two positive real-valued functions
  $12\sigma_1^2$ and $\rho^2\sigma_2^2$ must be zero at all arguments
  $s\in\Rb{4}$. This implies
  \begin{equation}
    \sigma_1=\basis{s_i}\sigma(s)=0
  \end{equation}
  for each $s\in\Rb{4}$ inside its domain and for each $i=1,2,3,4$. With this
  in mind, the equalities $\rho^2\kappa_{ij}(s,s) = 0$ for $i,j=1,2,3,4$ become
  \begin{equation}
    \begin{aligned}
      0 &= 12 + 3(\basis{s_i}\rho(s))^2 + 2\rho(s)((\basis{s_i})^2\rho(s))&&\text{for
      }i=j,\\
      0 &= 3(\basis{s_i}\rho(s))(\basis{s_j}\rho(s)) +
        2\rho(s)(\basis{s_i}\basis{s_j}\rho(s))&&\text{for }i\neq j.
    \end{aligned}
  \end{equation}
  Without loss of generality we assume that $\rho>0$. In this case
  by making the substitution $y(s) = \rho(s)^{5/2}$ the above system of
  differential equations for $\rho$ reduces to
  \begin{align}
    \label{eq:s2w-lagr-rays-main:y}
    0 &= (\basis{s_i})^2y(s) + 15y(s), \\
    0 &= \basis{s_i}\basis{s_j}y(s) \quad \text{for }i\neq j.
  \end{align}
  To see that this system has only trivial solutions, take the derivative of
  the first equation with respect to $s_j$ and of the second one with respect
  to $s_i$. We get
  \begin{equation}
    \begin{aligned}
      0 &= (\basis{s_i})^2\basis{s_j}y(s) + 15\basis{s_j}y(s), \\
      0 &= (\basis{s_i})^2\basis{s_j}y(s) \quad \text{for }i\neq j.
    \end{aligned}
  \end{equation}
  As a consequence we obtain $\basis{s_j}y(s)$ for each $j=1,2,3,4$. Therefore
  there exists $C\in\Rb{1}$ such that $y(s)=C$ at each $s\in\Rb{s-1}$.
  Inserting this into $(\ref{eq:s2w-lagr-rays-main:y})$ we obtain $y(s) = 0$.
  Since $y(s)=\rho(s)^{5/2}$, we obtain $\rho(s)=0$, which contradicts the fact
  that the hypersurfaces $H,K$ do not intersect. We have proved that the Ricci
  tensor of the bi-Lagrangian connection $\nabla$ cannot be everywhere zero in
  dimension $n=5$.
\end{proof}

This settles the existence problem in cases $n\neq 3$. The remaining case $n=3$
turns out to be more interesting.

\begin{thm}
  \label{thm:s2w-lagr-rays-main-n3}
  Let $n=3$. There exist pairs of disjoint surface germs $H,K$ in $\Rb{3}$
  which induce a bi-Lagrangian structure on the space of rays $T^*S^{2}$ that
  is Ricci flat. Such pairs $H,K$ are exactly the pairs of disjoint germs of a
  single sphere $S_{c,r}^2\subseteq\Rb{3}$ with arbitrary radius $r>0$ and
  center $c\in\Rb{3}$, that is,
  \begin{equation}
    S_{c,r}^2 = \smset{x\in\Rb{3} : \abs{x-c}=r},\qquad H=(S_{c,r}^2,x_0),
      \quad K=(S_{c,r}^2,y_0),
  \end{equation}
  for some $x_0,y_0\in S_{c,r}^2$ with $x_0\neq y_0$. Nevertheless, the induced
  bi-Lagrangian structure is never flat.
\end{thm}
\begin{proof}
  First, assume that the Ricci tensor of $\nabla$ vanishes. We proceed exactly
  as in the proof of Theorem $\ref{thm:s2w-lagr-rays-main}$, case $n=5$. Using
  the setup and notation established in the proof of said theorem, the equality
  $c_{00}=0$ and its derivative with respect to $s_i$ yield
  \begin{align}
    \label{eq:s2w-lagr-rays-main-n3:c00}
    0 &= 4 + \rho_1^2 + \rho\rho_2 - \sigma_1^2, \\
    0 &= 3\rho_1\rho_2 + \rho\rho_3 - 2\sigma_1\sigma_2^2.
  \end{align}
  The equality $c_{11}$ leads in turn to
  \begin{equation}
    \begin{aligned}
      0 &= -5 \rho^2 \sigma_2^2+\rho^2 \rho_2^2+\rho_1^2 (-8 \rho \rho_2-6
      \sigma_1^2+24)-4 \rho \rho_1 (\rho \rho_3-2 \sigma_1 \sigma_2) \\
      &\hskip 7em -4 \rho \rho_2 (\sigma_1^2-4)+3 \rho_1^4+3 \sigma_1^4-72
      \sigma_1^2+48.
    \end{aligned}
  \end{equation}
  If we let $q(s) = c_{11}(s)$, $p_1(s) = c_{00}(s)$ and $p_2(s) =
  \basis{s_i}c_{00}(s)$, then
  \begin{equation}
    q(s) = (60 + 3\rho_1^2 + \rho\rho_2 - 3\sigma_1^2)\cdot p_1(s) -
      (8\rho\rho_1)\cdot p_2(s) - 48(4+\rho_1^2+\rho\rho_2)-5\rho^2\sigma_2^2,
  \end{equation}
  from which we obtain
  \begin{equation}
    48(4+\rho_1^2+\rho\rho_2) + 5\rho^2\sigma_2^2 = 0.
  \end{equation}
  Inserting $(\ref{eq:s2w-lagr-rays-main-n3:c00})$ into this equation gives
  \begin{equation}
    48\sigma_1^2 + 5\rho^2\sigma_2^2 = 0.
  \end{equation}
  Since both summands are non-negative, this equality is equivalent to
  \begin{equation}
    \sigma_1 = \basis{s_i}\sigma(s) = 0.
  \end{equation}
  It holds irrespective of the choice of $i=1,2$, hence there exists
  $c_3\in\Rb{1}$ such that $\sigma(s) = 2c_3$ at each point $s\in\Rb{2}$. With
  this in mind, the equalities $\rho^2\kappa_{ij}(s,s) = 0$ for $i,j=1,2$
  involving the coefficients of the Ricci tensor $\mathrm{Rc} =
  \smsum{i,j}\kappa_{ij}\,ds_idt_j$ become
  \begin{equation}
    \begin{aligned}
      0 &= 4 + (\basis{s_i}\rho(s))^2 + \rho(s)((\basis{s_i})^2\rho(s)), \\
      0 &= (\basis{s_i}\rho(s))(\basis{s_j}\rho(s))
        + \rho(s)(\basis{s_i}\basis{s_j}\rho(s))\quad\text{for }i\neq j,
    \end{aligned}
  \end{equation}
  which reduce to
  \begin{equation}
    \begin{aligned}
      -8 &= (\basis{s_i})^2(\rho(s)^2), \\
      0 &= \basis{s_i}\basis{s_j}(\rho(s)^2)\quad\text{for }i\neq j.
    \end{aligned}
  \end{equation}
  The second of these equations tells us that $\rho(s)^2 = \rho_1(s_1) +
  \rho_2(s_2)$, while the first one establishes each $\rho_i(s_i)$ as a
  function of the form $\rho_i(s_i) = -4(s_i-c_i)^2 + b_i$ for some fixed
  $b_i,c_i\in\Rb{1}$. Therefore we can write
  \begin{equation}
    \begin{aligned}
      \rho(s)^2 &= -4\big((s_1-c_1)^2+(s_2-c_2)^2\big) + b
    \end{aligned}
  \end{equation}
  for some $c_1,c_2,b\in\Rb{1}$ with $b>0$. Let us write $b=4r^2$ for some
  $r>0$. Recall that $\rho(s) = f(s)-g(s)$ and $\sigma(s) = f(s)+g(s)$, where
  the functions $f(s),g(s)$ define the hypersurfaces $H,K$ respectively as
  graphs in $\Rb{3}$. Assume without loss of generality that $\rho(s)>0$.
  Solving for $f(s)$ and $g(s)$ yields
  \begin{equation}
    \label{eq:s2w-lagr-rays-main-fg}
    \begin{aligned}
      f(s) &= c_3 + \sqrt{r^2 -
        \big((s_1-c_1)^2+(s_2-c_2)^2\big)}, \\
      g(s) &= c_3 - \sqrt{r^2 -
        \big((s_1-c_1)^2+(s_2-c_2)^2\big)}.
    \end{aligned}
  \end{equation}
  Hypersurfaces given by graphs of these functions lie on a sphere
  \begin{equation}
    S_{c,r}^2 = \set{x\in\Rb{3} : (x_1-c_1)^2 + (x_2-c_2)^2
    + (x_3-c_3)^2 = r^2}
    \subseteq\Rb{3},
  \end{equation}
  where $c=(c_1,c_2,c_3)\in\Rb{3}$. These are the only
  possibilities for $H,K$ to induce a Ricci-flat bi-Lagrangian connection
  inside a system of coordinates normalized via Lemma
  $\ref{thm:s2w-lagr-rays-norm}$. In the original orthogonal coordinates, the
  set $S_{c,r}^2$ corresponds to an arbitrary sphere of positive
  radius, while the points of intersection $(s_1,s_2,f(s_1,s_2))$ and
  $(t_1,t_2,g(t_1,t_2))$ of the ray $(p,q)\in T^*S^2$ with
  $S_{c,r}^2\subseteq\Rb{3}$ correspond to any pair of different points
  $x_0,y_0\in S_{c,r}^2$. It is clear that $H=(S_{c,r}^2,x_0)$ and
  $K=(S_{c,r}^2,y_0)$ as surface-germs.

  It can be verified directly that the functions $f,g$ of the form
  $(\ref{eq:s2w-lagr-rays-main-fg})$ yield a Ricci-flat connection in dimension
  $n=3$. Indeed, by inserting them into $(\ref{eq:s2w-lagr-rays-norm:vol})$ we
  obtain
  \begin{equation}
    \begin{aligned}
      \omega^2 &= -h(s_1,s_2)h(t_1,t_2)
        \ ds_1\wedge ds_2\wedge dt_1\wedge dt_2,\quad \text{where}\\
        h(s_1,s_2) &=
        \frac{1}{\sqrt{r^2-\big((s_1-c_1)^2+(s_2-c_2)^2\big)}}.
    \end{aligned}
  \end{equation}
  For such a volume form the formula $(\ref{eq:s2w-lagr-rays:ricci})$ clearly
  evaluates to $0$. On the other hand, for the matrix $A=[a_{ij}]$ with entries
  given by $\omega = \smsum{i,j} a_{ij}\,ds_i\wedge dt_j$, the upper-left
  $2\times 2$ block $\Omega_\fol = d_y(d_x A\cdot A^{-1})^T$ of the matrix of
  bi-Lagrangian curvature $2$-forms $(\ref{eq:s2w-lagr-conn-coords-curv})$
  taken at $s=t=0$ evaluates to
  \begin{equation}
    \Omega_\fol = \frac{1}{4(r^2-c_1^2-c_2^2)^2} \begin{bmatrix}
        -r^2\,ds_1\wedge dt_1 + r^2\,ds_2\wedge dt_2
      & r^2\,ds_1\wedge dt_2 - 3r^2\,ds_2\wedge dt_1 \\
        -3r^2\,ds_1\wedge dt_2 + r^2\,ds_2\wedge dt_1
      & r^2\,ds_1\wedge dt_1 - r^2\,ds_2\wedge dt_2
    \end{bmatrix}.
  \end{equation}
  Since $r>0$, the curvature of the bi-Lagrangian connection $\nabla$ is not
  null despite $\nabla$ being Ricci-flat.
\end{proof}

\subsection{Bi-Lagrangian flatness of structures induced by tangents to
Lagrangian curves}
\label{ch:s2w-lagr-tangents}

Consider another example of a bi-Lagrangian structure provided by
Tabachnikov in \cite[{}I.6]{2-webs}.

Let $(\Rb{2n},\omega)$ be a standard symplectic space of dimension $2n$ with
canonical coordinates $(p_1,\ldots,p_n,q_1,\ldots,q_n)$ and symplectic form
$\omega=\sum_{i=1}^n dp_i\wedge dq_i$. To each Lagrangian submanifold
$L\subseteq\Rb{2n}$ one can associate a family $\fol_L$ of affine Lagrangian
subspaces $T_xL$ parametrized by points $x$ of $L$. If a point $p\in\Rb{2n}$
lies on the affine space $T_xL$ for some $x\in L$, then the family $\fol_L$ is
a foliation of the neighbourhood of $p$ if the contraction $\II_x(p-x,\cdot)$
of the second fundamental form $\II$ of $L$ with the affine vector $p-x$ is
invertible as a linear map from $T_xL$ to $T_xL^\bot$. Generically, two affine
Lagrangian subspaces intersect at a point $p\in\Rb{2n}$, hence for a pair of
generic Lagrangian submanifolds $L,K$ one obtains a bi-Lagrangian
structure $(U,\omega,\fol_L,\fol_K)$ defined on some neighbourhood $U$ of $p$
with foliations $\fol_L,\fol_K$ formed by the affine tangent spaces of $L,K$.

Tabachnikov encouraged his readers to find out which bi-Lagrangian structures
of this kind are trivial. We were able to solve this problem in the
$2$-dimensonal case, where both Lagrangian submanifolds are $L,K$ regular
curves, under a natural assumption of \emph{regularity} of the structure
induced by tangents: we require that for each point $p_0$ of its domain $U$ and
points $p_1\in L$, $p_2\in K$ such that $p_0\in T_{p_1}L\cap T_{p_2}K$, the
tangents to the restrictions $L_{|V_1},K_{|V_2}$ of $L,K$ to arbitrary
open neighbourhoods $V_1\subseteq L$, $V_2\subseteq K$ of $p_1,p_2$ induce a
bi-Lagrangian structure on some open neighbourhood of $p_0$ (or, in other
words, that the map $(p_1,p_2)\mapsto p_0$ from points of tangencies to $L,K$
to the intersection of the corresponding tangents in $U$ is open). We used
methods similar to those used in Section $\ref{ch:s2w-lagr-rays}$, where the
bulk of the argument rests upon computer-assisted calculations. The authors
themselves have relied on \texttt{Wolfram Mathematica 13} \cite{math} to obtain
their result. It states that the curvature of the canonical connection cannot
vanish identically for any regular bi-Lagrangian structure of the above kind.
Before proving this theorem, we state the conditions for genericity and
regularity of the structure in question in the form of a lemma.

\begin{lem}
  \label{thm:s2w-lagr-tangents-coords}
  Let $\omega_0=dx\wedge dy$ be the germ of the standard symplectic form on
  $\Rb{2}$ at $p_0=(x_0,y_0)\in\Rb{2}$, and let $L,K$ be two germs of curves at
  points $p_1=(x_{01},y_{01}), p_2=(x_{02},y_{02})\in\Rb{2}$ respectively. The
  quadruple $(\Rb{2},\omega_0,\fol_K,\fol_L)$, where $\fol_L=\smset{T_qL:q\in
  L}, \fol_K=\smset{T_qK:q\in K}$ forms a bi-Lagrangian structure-germ at $p_0$
  if the following three conditions hold:
  \begin{enumerate}[label=$(\alph{enumi})$, ref=\alph{enumi}]
    \item\label{thm:s2w-lagr-tangents-2d:fols}
      the affine lines $T_{p_1}L$ and $T_{p_2}K$ intersect transversely at
      point $p_0=(x_0,y_0)$,
    \item\label{thm:s2w-lagr-tangents-2d:graphs}
      the affine line containing $p_1,p_2$ intersects both $L$ and $K$
      transversely,
    \item\label{thm:s2w-lagr-tangents-2d:non-sing}
      the curvatures of $L,K$ at $p_1,p_2$ respectively are non-zero.
  \end{enumerate}
  Moreover, if all of the above conditions are met, then
  any parametrization $\gamma_L(s)=(x_1(s),y_1(s))$ of $L$ and
  $\gamma_K(t)=(x_2(t),y_2(t))$ of $K$ with
  $\gamma_L(0)=(x_{01},y_{01}),\gamma_K(0)=(x_{02},y_{02})$ yields a local
  coordinate system $(s,t)\mapsto p(s,t)$ satisfying $T\fol_L=\ker ds$ and
  $T\fol_K=\ker dt$, where $\smset{p}=T_{(x_1,y_1)}L\cap T_{(x_2,y_2)}K$. In
  this coordinate system the symplectic form $\omega$ is
  \newcommand{\gl}{\gamma_L}
  \newcommand{\gk}{\gamma_K}
  \begin{equation}
    \label{eq:s2w-lagr-tangents-2d:vol}
    \begin{aligned}
      &dx\wedge dy =
        \det\big(\gl'(s),\gl(s)-\gk(t)\big)
        \det\big(\gk'(t),\gl(s)-\gk(t)\big) \\
        & \hskip 5em\cdot
        \det\big(\gl'(s),\gl''(s)\big)
        \det\big(\gk'(t),\gk''(t)\big)
        \det(\gl'(s),\gk'(t))^{-3}\ ds\wedge dt,
    \end{aligned}
  \end{equation}
  where $\det(v,w) = v_1w_2-v_2w_1$ for each pair of vectors $v=(v_1,v_2),
  w=(w_1,w_2)\in\Rb{2}$.
\end{lem}
\begin{proof}
  \newcommand{\gl}{\gamma_L}
  \newcommand{\gk}{\gamma_K}
  Assume first that $(\ref{thm:s2w-lagr-tangents-2d:fols})$,
  $(\ref{thm:s2w-lagr-tangents-2d:graphs})$,
  $(\ref{thm:s2w-lagr-tangents-2d:non-sing})$ hold and express $L,K$ as images
  of some parametrized curve-germs $\gl(s)=(x_1(s),y_1(s))$ and
  $\gk(t)=(x_2(t),y_2(t))$ at $0$. Since the affine tangents $T_{\gamma_L(s)}L$
  and $T_{\gamma_K(t)}K$ have nonempty intersection by
  $(\ref{thm:s2w-lagr-tangents-2d:fols})$ and continuity, their unique common
  point $p=(x,y)$ satisfies $p-\gamma_L(s)\in T_{p_1}L=\big\langle\mskip 1mu
  (x_1'(s),y_1'(s))\mskip 1mu\big\rangle$ and $p-\gamma_K(t)\in
  T_{p_2}K=\big\langle\mskip 1mu (x_2'(t),y_2'(t))\mskip 1mu\big\rangle$. This
  translates to the following linear system of equations.
  \begin{equation}
    \left\{\begin{aligned}\mskip 4mu
      \big(x-x_1(s)\big)y_1'(s) - \big(y-y_1(s)\big)x_1'(s) &= 0, \\
      \big(x-x_2(t)\big)y_2'(t) - \big(y-y_2(t)\big)x_2'(t) &= 0.
    \end{aligned}\right.
  \end{equation}
  Its solution,
  \begin{equation}
    \label{eq:s2w-lagr-tangents-2d:param}
    \left\{\begin{aligned}\mskip 4mu
      x(s,t) &= \frac{
        \big(y_1(s)-y_2(t)\big)x_1'(s)x_2'(t)+x_1'(s)x_2(t)y_2'(t)-x_1(s)x_2'(t)y_1'(t)
      }{x_1'(s)y_2'(t)-x_2'(t)y_1'(s)}, \\[1ex]
      y(s,t) &= -\frac{
        \big(x_1(s)-x_2(t)\big)y_1'(s)y_2'(t)+y_1'(s)y_2(t)x_2'(t)-y_1(s)y_2'(t)x_1'(t)
      }{x_1'(s)y_2'(t)-x_2'(t)y_1'(s)},
    \end{aligned}\right.
  \end{equation}
  expresses $(x,y)$ as a function of parameters $s,t$. From this it is
  straightforward to compute $dx\wedge dy$ in terms of $ds,dt$. The result is
  the $2$-form $(\ref{eq:s2w-lagr-tangents-2d:vol})$. It is well-defined and
  nondegenerate for $(s,t)$ in a small neighbourhood of $0$. To see this, note
  that the first two factors $\det(\gl'(s),\gl(s)-\gk(t))$ and
  $\det(\gk'(t),\gl(s)-\gk(t))$ of $(\ref{eq:s2w-lagr-tangents-2d:vol})$ are
  nonvanishing by continuity and assumption
  $(\ref{thm:s2w-lagr-tangents-2d:graphs})$ of the theorem. The next two,
  namely $\det\big(\gl'(s),\gl''(s)\big)$ and $\det\big(\gk'(t),\gk''(t)\big)$,
  are exactly the curvatures of $\gamma_L,\gamma_K$ at
  $\gamma_L(s),\gamma_K(t)$ respectively, hence are non-zero by
  $(\ref{thm:s2w-lagr-tangents-2d:non-sing})$ and continuity. Finally,
  nonvanishing of the last factor $\det(\gl'(s),\gk'(t))$ follows from
  $(\ref{thm:s2w-lagr-tangents-2d:fols})$. This proves $(s,t)\mapsto
  \big(x(s,t),y(s,t)\big)$ is a valid local coordinate system. Moreover it
  satisfies $T\fol_L=\ker ds$ and $T\fol_K=\ker dt$ by construction, hence both
  $\fol_L, \fol_K$ are foliations of a neighbourhood of $p_0$.
\end{proof}

Let $(U,\omega,\fol_L,\fol_K)$ be any regular bi-Lagrangian structure induced
by tangents to $L,K$ on an open set $U\subseteq\Rb{2}$. The conditions
$(\ref{thm:s2w-lagr-tangents-2d:fols})$,
$(\ref{thm:s2w-lagr-tangents-2d:graphs})$ and
$(\ref{thm:s2w-lagr-tangents-2d:non-sing})$ as stated in Lemma
$\ref{thm:s2w-lagr-tangents-coords}$ are indeed satisfied at generic points of
$U$, or, more precisely, the set of points $p_0\in U$ such that there exist
points $p_1\in L$ and $p_2\in K$ with $p_0\in T_{p_1}L\cap T_{p_2}K$ satisfying
$(\ref{thm:s2w-lagr-tangents-2d:fols})$,
$(\ref{thm:s2w-lagr-tangents-2d:graphs})$ and
$(\ref{thm:s2w-lagr-tangents-2d:non-sing})$ is open and dense in $U$.

To see this, assume first that $(\ref{thm:s2w-lagr-tangents-2d:fols})$ does not
hold at a certain point $p_0\in T_{p_1}L\cap T_{p_2}K$. Then
$T_{p_1}L=T_{p_2}K$ share a leaf, hence $(U,\omega,\fol_L,\fol_K\!)$ is not a
bi-Lagrangian structure, a contradiction. The conditions
$(\ref{thm:s2w-lagr-tangents-2d:graphs})$ and
$(\ref{thm:s2w-lagr-tangents-2d:non-sing})$, which concern pairs of curves
$L,K$, are conjunctions of two sub-conditions
$(\ref{thm:s2w-lagr-tangents-2d:graphs}_L\!)$,
$(\ref{thm:s2w-lagr-tangents-2d:graphs}_K\!)$ and
$(\ref{thm:s2w-lagr-tangents-2d:non-sing}_L\!)$,
$(\ref{thm:s2w-lagr-tangents-2d:non-sing}_K\!)$ regarding the individual curves
$L$, $K$ in a natural way. Since all of the above conditions are open in
$L\times K$, the sets of points $p_0\in U$ such that these conditions are
satisfied for some $p_1\in L$ and $p_2\in K$ is also open in $U$ due to
regularity of the structure. Therefore the only thing left to check is their
density, since finite intersections of open and dense subsets are also open and
dense. Assume now that the condition
$(\ref{thm:s2w-lagr-tangents-2d:graphs}_L\!)$ does not hold on some
neighbourhood of $p_0\in U$ with $p_0\in T_{p_1}L\cap T_{p_2}K$, that is, the
affine line $\ell$ containing $p_L,p_K$ does not intersect $L$ transversely for
$p_L\in L,p_K\in K$ close to $p_1,p_2$ repsectively. Then $p_1-p_K\in T_{p_1}L$
for each $p_K$ close to $p_2$, hence an open subset of $K$ containing $p_2$ is
contained in a line $T_{p_1}L$, so that $T_{p_2}K=T_{p_1}L$ contradicting
property $(\ref{thm:s2w-lagr-tangents-2d:fols})$ established above for all
points $p_0\in U$. To obtain the density of condition
$(\ref{thm:s2w-lagr-tangents-2d:non-sing}_L\!)$ note that if the curvature of
$L$ vanishes at all $p_L\in L$ in some neighbourhood $V$ of $p_1$, then
$L_{|V}$ is a fragment of a line, hence all tangents of $L_{|V}$ coincide and
$\fol_{L_{|V}}$ does not form a foliation of a neighbourhood of any point
$p_0\in U$ satisfying $p_0\in T_{p_1}L\cap T_{p_2}K$, which contradicts
regularity of the structure. The density claims proved above are mirrored in
the corresponding claims for $(\ref{thm:s2w-lagr-tangents-2d:graphs}_K\!)$ and
$(\ref{thm:s2w-lagr-tangents-2d:non-sing}_K\!)$, which together yield the
desired result.

With this genericity claim in mind, it is possible to reduce the global problem
to its localized, generic version, the formulation of which is the content of
Theorem $\ref{thm:s2w-lagr-tangents-2d}$ below.

\begin{thm}
  \label{thm:s2w-lagr-tangents-2d}
  Let $\omega_0=dx\wedge dy$ be the standard symplectic form on $\Rb{2}$,
  let $L,K$ be two germs of curves at points $p_1=(x_1,y_1),
  p_2=(x_2,y_2)\in\Rb{2}$ respectively and assume that conditions
  $(\ref{thm:s2w-lagr-tangents-2d:fols})$,
  $(\ref{thm:s2w-lagr-tangents-2d:graphs})$,
  $(\ref{thm:s2w-lagr-tangents-2d:non-sing})$ of Lemma
  $\ref{thm:s2w-lagr-tangents-coords}$ hold, so that
  $(\Rb{2},\omega_0,\fol_L,\fol_K)$ is a bi-Lagrangian structure-germ at
  $p_0\in\Rb{2}$ with foliations $\fol_L=\smset{T_qL:q\in L},
  \fol_K=\smset{T_qK:q\in K}$. The canonical connection $\nabla$ of the
  bi-Lagrangian structure-germ $(\Rb{2},\omega_0,\fol_L,\fol_K)$ is never flat.
\end{thm}
\begin{proof}
  \newcommand{\gl}{\gamma_L}
  \newcommand{\gk}{\gamma_K}
  Fix a parametrization of $L,K$ by parametrized curve-germs
  $\gl(s)=(x_1(s),y_1(s))$ and $\gk(t)=(x_2(t),y_2(t))$ at $0$. Recall that, by
  Lemma $\ref{thm:s2w-lagr-tangents-coords}$, the map $(s,t)\mapsto
  p(s,t)=\big(x(s,t),y(s,t)\big)$ for $\smset{p}=T_{(x_1,y_1)}L\cap
  T_{(x_2,y_2)}K$ is a valid local coordinate system satisfying $T\fol_L=\ker
  ds$ and $T\fol_K=\ker dt$, hence we can apply
  $(\ref{eq:s2w-lagr-conn-coords-curv})$ to the expression
  $(\ref{eq:s2w-lagr-tangents-2d:vol})$ for the symplectic form $\omega_0$ to
  compute the curvature of the bi-Lagrangian connection $\nabla$. The only
  independent coefficient $\kappa$ of the curvature $2$-form $\Omega$ is given
  by
  \begin{equation}
    \begin{aligned}
      \kappa &= \smfrac{
        \det\big(\gl'(s),\gk'(t)\big)
        \det\big(\gl''(s),\gl(s)-\gk(t)\big)
      }{\det\big(\gl'(s),\gl(s)-\gk(t)\big)^2}
      -
      \smfrac{
        \det\big(\gl''(s),\gk'(t)\big)
      }{\det\big(\gl'(s),\gl(s)-\gk(t)\big)} \\
      &{}+
      \smfrac{
        \det\big(\gl'(s),\gk'(t)\big)
        \det\big(\gk''(t),\gl(s)-\gk(t)\big)
      }{\det\big(\gk'(t),\gl(s)-\gk(t)\big)^2}
      -
      \smfrac{
        \det\big(\gl'(s),\gk''(t)\big)
      }{\det\big(\gk'(t),\gl(s)-\gk(t)\big)} \\
      &{}+
      3\smfrac{
        \det\big(\gl''(s),\gk'(t)\big)
        \det\big(\gl'(s),\gk''(t)\big)
      }{\det\big(\gl'(s),\gk'(t)\big)^2}
      -
      3\smfrac{
        \det\big(\gl''(s),\gk''(t)\big)
      }{\det\big(\gl'(s),\gk'(t)\big)}.
    \end{aligned}
  \end{equation}

  From now on, we will proceed as in the proof of Theorems
  $\ref{thm:s2w-lagr-rays-main}$ and $\ref{thm:s2w-lagr-rays-main-n3}$. Assume
  to the contrary that $\kappa=0$ everywhere. First, use rigid motions
  $R\in\Rb{2}\rtimes \mathrm{SO}(2)$ to simplify the problem by choosing an
  orthogonal coordinate system $(x,y)$ in which $p_1=(0,0)$ and $p_2-p_1=(0,a)$ for
  some $a\in\Rb{1}$. Assumption $(\ref{thm:s2w-lagr-tangents-2d:graphs})$ means
  that the curves $L,K$ can be expressed as the images of
  $\gamma_L(s)=(s,f(s))$ and $\gamma_K(t)=(t,g(t))$ for some smooth
  function-germs $f,g\in C^\infty(\Rb{1},0)$. Introduce the following notation:
  for each $j,k\in\mathbb{N}$ put
  \begin{equation}
    \tilde{c}_{jk}(s)=\frac{\partial^{j+k}\kappa}{\partial s^j\partial
    t^k}_{|(s,s)}
  \end{equation}
  and
  \begin{equation}
    \begin{aligned}
      \sigma(s)&=f(s)+g(s),&\qquad
      \sigma_j(s) &= \smfrac{\partial^j\sigma}{\partial s^j}(s),\\
      \rho(s)&=f(s)-g(s),&\qquad
      \rho_k(s) &= \smfrac{\partial^k\rho}{\partial s^k}(s).
    \end{aligned}
  \end{equation}
  as in Theorem $\ref{thm:s2w-lagr-rays-main}$. The assumptions
  $(\ref{thm:s2w-lagr-tangents-2d:fols})$,
  $(\ref{thm:s2w-lagr-tangents-2d:graphs})$ and
  $(\ref{thm:s2w-lagr-tangents-2d:non-sing})$ for $s=t$ correspond to $\rho\neq
  0$, $\rho_1\neq 0$ and $\rho_2^2\neq\sigma_2^2$ respectively. Now, $\kappa=0$
  implies $\tilde{c}_{jk}=0$ for each $j,k\in\mathbb{N}$. Since the
  denominators of all $\tilde{c}_{jk}$ are products of $\rho$, $\rho_1$ and
  $\rho_2^2-\sigma_2^2$ by $(\ref{eq:s2w-lagr-tangents-2d:vol})$, we can put
  $\tilde{c}_{jk}$ in their common denominator forms and consider only the
  equalities given by setting their numerators $c_{jk}$ zero. Some of the first
  few equalities of this kind are
  \begin{equation}
    \begin{aligned}
      0 &= c_{00} = 4\rho_1^2\rho_2 + 3\rho(\rho_2^2-\sigma_2^2), \\
      0 &= c_{10} = 3 \rho^2 \rho_1 \big(\rho_2-\sigma_2\big)
      \big(\rho_3+\sigma_3\big)-3 \rho^2 \big(\rho_2-\sigma_2\big)
      \big(\rho_2+\sigma_2\big){}^2 \\
      &\hskip 4em -2 \rho_1^4 \big(\rho_2-3 \sigma_2\big)+2 \rho \rho_1^3
      \big(\rho_3+\sigma_3\big), \\
      0 &= c_{01} = 3 \rho^2 \rho_1 \big(\rho_2+\sigma_2\big)
      \big(\rho_3-\sigma_3\big)-3 \rho^2 \big(\rho_2-\sigma_2\big){}^2
      \big(\rho_2+\sigma_2\big)\\
      &\hskip 4em -2 \rho_1^4 \big(\rho_2+3 \sigma_2\big)+2
      \rho \rho_1^3 \big(\rho_3-\sigma_3\big), \\
      0 &= c_{11} = 6 \rho^3 \rho_1^2 \big(\rho_3^2-\sigma_3^2\big)-12 \rho^3 \rho_1 \big(-2
      \rho_2 \sigma_2 \sigma_3+\rho_3 \sigma_2^2+\rho_2^2\rho_3\big)\\
      & \hskip 4em +9 \rho^3
      \big(\rho_2^2-\sigma_2^2\big){}^2-4 \rho \rho_1^4 \big(\rho_2^2+3
      \sigma_2^2\big) -16 \rho_2 \rho_1^6-16 \rho \rho_1^5 \rho_3.
    \end{aligned}
  \end{equation}

  Before proceeding, we offer two remarks about the above expressions. First,
  note that $c_{00}=0$ implies that $\rho_2$ does not vanish. If this were the
  case, we would have $3\rho\sigma_2^2 = 0$, hence $\sigma_2=0=\rho_2$ for some
  $s$ near $0$, which contradicts assumption
  $(\ref{thm:s2w-lagr-tangents-2d:non-sing})$. For the second remark, a~quick
  glance at $c_{jk}$ reveals that all of these expressions are polynomials in
  variables $\rho_i$ and $\sigma_i$ for $i=0,1,2,\ldots,\max(j,k)+2$.

  We now proceed to an elementary argument that the equalities $c_{jk}=0$ for
  $j,k=0,1,2$ lead to a contradiction. It is very difficult to perform the
  necessary calculations by hand, hence it is well-advised to verify our
  reasoning using a computer algebra system.

  By eliminating $\sigma_2,\sigma_3,\rho_3$ from the system of
  equalities $c_{00}=0$, $c_{01}-c_{10}=0$, $c_{11}=0$ and $\basis{s}c_{00}=0$
  under the assumptions $\rho_i\neq 0$ for $i=0,1,2$ and
  $\rho_2^2-\sigma_2^2\neq 0$, we arrive at the equality
  \begin{equation}
    (3\rho_1^2+2\rho\rho_2)(4\rho_1^2+3\rho\rho_2)=0,
  \end{equation}
  which holds everywhere. Hence, for each parameter $s$ sufficiently close to
  $0$, either $(1)$ $3\rho_1^2+2\rho\rho_2=0$, or $(2)$
  $4\rho_1^2+3\rho\rho_2$. If case $(1)$ holds for some fixed $s_0\in\Rb{1}$,
  eliminate $\rho_2(s_0),\rho_3(s_0),\rho_4(s_0),$
  $\sigma_2(s_0),\sigma_3(s_0), \sigma_4(s_0)$ from the following system of 8
  equalities
  \begin{equation}
    \left\{\hskip 1ex \begin{aligned}
      c_{02}(s_0)+c_{20}(s_0)&=0, & c_{00}(s_0) &= 0, & (\basis{s}c_{00})(s_0) &= 0, \\
      c_{02}(s_0)-c_{20}(s_0)&=0, & c_{11}(s_0) &= 0, & c_{01}(s_0)-c_{10}(s_0) &= 0, \\
      c_{21}(s_0)+c_{12}(s_0)&=0, & && c_{21}(s_0) - c_{12}(s_0) &= 0,
    \end{aligned}\right.
  \end{equation}
  to reach $\rho_1(s_0)^{12} = 0$; a contradiction. In case $(2)$, a similar
  variable elimination leads to $\rho_1(s_0)^8=0$, which is also a
  contradiction. Both of these together show that $\kappa=0$ is impossible.
  This concludes the proof.
\end{proof}

  \printbibliography

\end{document}